\documentclass{sn-jnl}
\usepackage{algorithm}%
\usepackage{algorithmicx}%
\usepackage{algpseudocode}%
\usepackage{amsmath, amsfonts, amssymb, amsthm, bm, stmaryrd,  mathtools}
\usepackage{enumitem}
\usepackage[title]{appendix}%
\usepackage{booktabs}%
\usepackage{comment}
\usepackage{graphicx} 
\usepackage{listings}%
\usepackage{multirow}%
\usepackage{mathrsfs}%
\usepackage{manyfoot}%
\usepackage{textcomp}%
\usepackage{tikz}
\usepackage{tikz-cd}
\usepackage{times}
\usepackage{xcolor}%
\usepackage{xspace}

\tikzset{commutative diagrams/.cd}
\tikzcdset{row sep/tiny=0.12 em}

\usetikzlibrary{intersections}

\newlength\myheight

        \usetikzlibrary{calc}

\theoremstyle{thmstyleone}%

\newtheorem{Def}{Definition}[section]
\newenvironment{definition}{\begin{Def} \rm}{\end{Def}}
\newtheorem{lemma}[Def]{Lemma}
\newtheorem{proposition}[Def]{Proposition}
\newtheorem{corollary}[Def]{Corollary}
\newtheorem{theorem}[Def]{Theorem}
\newtheorem{example}[Def]{Example}
\newtheorem{remark}[Def]{Remark}

\newtheorem{Construction}[Def]{Construction}
\newenvironment{construction}{\begin{Construction} \rm}{\end{Construction}}

\newtheorem{claim}{Claim}[Def]

\newcommand{\set}[2]{\{#1\;|\;#2\}}

\newcommand{\conjun}{\mathrel{\wedge}}

\newcommand{\Cat}[1]{\ensuremath{\mathbf{#1}}}

\newcommand{\sai}[1]{[\,#1\,]}
\newcommand{\EndoHom}[1]{\mathcal{E}{\kern-.5ex}\textit{n}{\kern-.2ex}\textit{d}{\kern-.2ex}\textit{o}(#1)}

\newcommand{\sqdot}{\begin{tikzpicture}[baseline={(current bounding box.south)}, scale=0.3]
    \draw[line width=0.5pt] (-0.3, -0.3) rectangle (0.3, 0.3);
    \fill (0, 0) circle (1.5pt);
\end{tikzpicture}}

\newcommand{\sqtilde}{\begin{tikzpicture}[baseline={(current bounding box.south)}, scale=0.3] 
    \draw[line width=0.5pt] (-0.35, -0.35) rectangle (0.35, 0.35);
    
    \node[inner sep=0pt, outer sep=0pt] at (0,0) {$\thicksim$}; 
\end{tikzpicture}}

\begin{document}

\title[On $\mathscr{T}$-based orthomodular dynamic algebras]{On $\mathscr{T}$-based orthomodular dynamic algebras}


\author*[1]{\fnm{Jan} \sur{Paseka}}\email{paseka@math.muni.cz}
\equalcont{These authors contributed equally to this work.}

\author[1,2]{\fnm{Juanda Kelana} \sur{Putra}}\email{juandakelanaputra@walisongo.ac.id}
\equalcont{These authors contributed equally to this work.}

\author[1]{\fnm{Richard} \sur{Smolka}}\email{394121@mail.muni.cz}
\equalcont{These authors contributed equally to this work.}

\affil*[1]{\orgdiv{Department of Mathematics and Statistics}, \orgname{Masaryk University}, \orgaddress{\street{Kotl\'a\v rsk\'a 2}, \city{611\,37 Brno}, 
\country{Czech Republic}}}

\affil[2]{\orgdiv{Department of Mathematics}, \orgname{%
Walisongo State Islamic  University}, 
\country{Indonesia}}



\abstract{This paper establishes a categorical equivalence between the category
$\mathbb{COL}$ of complete orthomodular lattices and the category 
$\mathscr{T}\mathbb{ODA}$ of $\mathscr{T}$-based orthomodular dynamic algebras.
Complete orthomodular lattices serve as the static algebraic foundation for
quantum logic, modeling the testable properties of quantum systems. In
contrast, $\mathscr{T}$-based orthomodular dynamic algebras, which are
specialized unital involutive quantales, formalize the composition and
quantum-logical properties of quantum actions.

This result refines prior connections between orthomodular lattices and dynamic algebras, provides a constructive bridge between static and dynamic quantum logic perspectives, and extends naturally to Hilbert lattices and broader quantum-theoretic structures.}

\keywords{Orthomodular dynamic algebra, complete orthomodular lattice, quantale, quantale module}



\maketitle

\section{Introduction}\label{sec1}

Quantum mechanics fundamentally challenges classical
intuitions about logic and measurement. Unlike classical systems, which are
governed by Boolean logic, quantum systems adhere to non-Boolean structures.
This discrepancy motivates the development of mathematical frameworks capable
of capturing both quantum propositions and their dynamics.

Two complementary perspectives have emerged in this endeavor. The first, rooted in the work of Birkhoff and von Neumann, employs complete orthomodular lattices to formalize the testable properties of quantum systems. These lattices provide a static, algebraic foundation for quantum logic, encoding the structure of propositions and their orthogonality relations. The second perspective, developed more recently through dynamic epistemic logic and quantum dynamic algebras, emphasizes the operational and transformational aspects of quantum theory---how quantum actions compose, interact, and modify the state of a system.

The dynamic approach to quantum logic was pioneered by 
\cite{BaltagSmets2005}. They introduced quantum dynamic logic to reason about measurements, operations, and information flow, extending classical dynamic logic to the non‑Boolean quantum setting. Building on this foundation, 
\cite{KRSZ} developed the theory of orthomodular dynamic algebras, providing an algebraic counterpart to Baltag and Smets' logical framework and establishing initial connections to orthomodular lattices.

This paper establishes a precise categorical equivalence between these two viewpoints. Specifically, we demonstrate that the category of complete orthomodular lattices is equivalent to the category of $\mathscr{T}$-based orthomodular dynamic algebras, a specialized class of unital involutive quantales designed to formalize quantum actions. The equivalence shows that the static lattice-theoretic view and the dynamic, quantale-based view are categorically equivalent, providing a direct bridge between these perspectives.

Our approach builds upon and refines the work of 
\cite{KRSZ}, which established initial links between orthomodular lattices and dynamic algebras. We extend their results by introducing a more flexible framework based on involutive submonoids of the Foulis quantale of linear maps. For a complete orthomodular lattice $\mathcal{M} = (M, \leq, (-)^\perp)$, we construct an involutive submonoid $L_{\mathcal{M}}$ satisfying 
\[
\{\pi_m \mid m \in M\} \subseteq L_{\mathcal{M}} \subseteq \mathbf{Lin}(M),
\]
where $\pi_m$ denotes the Sasaki projection onto $m$ and 
$\mathbf{Lin}(\mathcal M)$ is the Foulis quantale of linear endomaps on 
$\mathcal{M}$. This intermediate choice between Sasaki projections and all linear maps provides fine-grained control over the resulting dynamic algebra structure while maintaining full compatibility with the orthomodular lattice operations.

The free unital involutive quantale $\mathscr{P}(L_{\mathcal{M}})$ of all subsets of $L_{\mathcal{M}}$, equipped with setwise composition and an appropriately defined orthocomplement operator ${\sim}$, yields a $\mathscr{T}$-based orthomodular dynamic algebra whose ``test set'' 
is naturally identified with the elements of the original lattice $\mathcal{M}$. This construction is functorial, and we establish that the resulting functors between the categories $\mathbb{COL}$ and $\mathscr{T}\mathbb{ODA}$ form an equivalence via natural isomorphisms in both directions. In doing so, we clarify how dynamics emerge from orthomodular lattices and strengthen the connection with unital involutive quantales, providing a comprehensive algebraic foundation for the dynamic quantum logic initiated by Baltag and Smets.

The significance of this equivalence extends beyond pure category theory. Complete orthomodular lattices encompass several important classes of quantum-logical structures, including Hilbert lattices---the lattices of closed subspaces of Hilbert spaces that arise directly from the mathematical formalism of quantum mechanics. Our equivalence therefore provides a robust categorical link between quantales and a broad spectrum of quantum-theoretic frameworks, from abstract quantum logic to concrete operator algebras. This connection facilitates the transfer of results, techniques, and intuitions between these traditionally distinct areas of study, and demonstrates that the dynamic perspective on quantum logic advocated by Baltag and Smets is not merely a convenient formalism but is categorically equivalent to the classical static perspective.

Methodologically, our contribution lies in providing an explicit, constructive proof of this equivalence that emphasizes conceptual clarity and computational feasibility. By working directly within the Foulis quantale $\mathbf{Lin}(\mathcal{M})$ and carefully controlling the generator scheme $L_{\mathcal{M}}$, we obtain a streamlined demonstration of minimality and a transparent interaction between the orthocomplement operator ${\sim}$ and Sasaki projections. The approach simplifies prior arguments and facilitates computational implementation and extensions to broader structures.

The paper is organized as follows. Section~\ref{sec:preliminaries} reviews the necessary background on quantales, involutive structures, Foulis quantales, and complete orthomodular lattices, establishing notation and recalling essential results. Section~\ref{igedyal} introduces involutive generalized dynamic algebras and develops their fundamental properties, including the crucial module structure that connects dynamics to the test set. Section~\ref{sectionTODA} defines the category $\mathscr{T}\mathbb{ODA}$ of $\mathscr{T}$-based orthomodular dynamic algebras and establishes their key structural characteristics. Subsequently, 
we provide a concrete instantiation of the abstract functor $\mathscr{T}$, demonstrating how our framework subsumes and generalizes prior results in the literature. Section~\ref{sec:construction-gamma} constructs the functor $\Gamma: \mathbb{COL} \to \mathscr{T}\mathbb{ODA}$ that transforms orthomodular lattices into dynamic algebras. Section~\ref{sec:functor-psi} presents the reverse functor $\Psi: \mathscr{T}\mathbb{ODA} \to \mathbb{COL}$ and proves the natural isomorphisms that complete the categorical equivalence. 

Throughout, we assume familiarity with basic category theory, lattice theory, and the fundamentals of quantum logic. Readers seeking additional background on quantales and orthomodular lattices are directed to the comprehensive treatments by \cite{KP} and \cite{Kalmbach83}, respectively.

\section{Preliminaries}\label{sec:preliminaries}

The initial phase of this work entails the standardization of notation and a review of the foundational structures. 

\subsection{Quantales and related structures}

We begin by reviewing the concepts of involutive semigroups,
quantales, and related structures. Quantales provide a dynamic perspective for
analyzing complex systems, enabling us to model the evolution of a system and
the structure of quantum actions that govern these changes.

\medskip 
\begin{definition}[Semigroup involution]
Let $(S,\odot)$ be a semigroup. An \emph{involution} on $S$ is a map ${}^{*}:S\to S$ satisfying the following properties for all $x,y\in S$:
\begin{itemize}
    \item $(x^{*})^{*} = x$,
    \item $(x \odot y)^{*} = y^{*} \odot x^{*}$.
\end{itemize}

A semigroup equipped with an involution is called an \emph{involutive semigroup}.
\end{definition}

\medskip
\begin{definition}
A \emph{quantale} is a tuple $\mathcal{Q}=(Q,\bigsqcup,\odot)$, where $Q$ 
is a complete join-semilattice and $\odot$ is a binary operation
on $Q$ satisfying the following properties for all $a,b,c\in Q$ and 
$S\subseteq Q$: 
\begin{enumerate}[label=(\textbf{Q\arabic*}),leftmargin=1.57905cm] \item \emph{Associativity:} $a\odot(b\odot c)
= (a\odot b)\odot c$, 
\item \emph{Left distributivity:} $a\odot(\bigsqcup S) =
\bigsqcup_{s\in S}(a\odot s)$, 
\item \emph{Right distributivity:} $(\bigsqcup
S)\odot a = \bigsqcup_{s\in S}(s\odot a)$. 
\end{enumerate}

A quantale $\mathcal{Q}$ is called {\em unital} if there exists an element $e\in Q$ such that for every $a\in Q$ the equalities $a\odot e=a$ and $e\odot a=a$ hold. We say that $e$ is a {\em unit} of $\mathcal{Q}$. A {\em quantale congruence } on a 
quantale $\mathcal{Q}$ is an equivalence on $Q$ that preserves 
arbitrary joins and multiplication.

By an {\em involutive quantale} we mean a quantale 
$\mathcal{Q}$ equipped with a semigroup involution \(^{*}\) on $Q$ satisfying
\[
\Bigl(\bigsqcup_{i\in I} a_i\Bigr)^{*} = \bigsqcup_{i\in I} a_i^{*}
\]
for all \(a_i\in Q, i \in I\), where $I$ is an arbitrary index set \emph{(possibly empty)}. In particular, we adopt the standard convention that the least element and the greatest element of the complete join-semilattice $Q$ are defined as the empty join and the join of all elements, respectively:
\[\bigsqcup_{i\in \varnothing} a_i \;=\; \bigsqcup \varnothing \;=:\; 0\quad
\text{ and }\quad \bigsqcup Q \;=:\; 1.
\]

Hence, the above definition yields

\[
0^{*}=0\quad
\text{ and }\quad 1^{*}=1.
\]

Moreover, if an involutive quantale is unital with unit $e$ then 
$e=e^{*}$.

A {\em homomorphism of unital involutive quantales} is a 
map $h\colon Q_1 \to Q_2$ between unital involutive quantales 
that preserves arbitrary joins, multiplication, unit and involution.
\end{definition}
\medskip
\begin{remark}
Although we assume that $(Q,\bigsqcup)$ is a complete join-semilattice, this
already determines a complete lattice order by $a\sqsubseteq b \iff a\bigsqcup b=b$.
With respect to this order, arbitrary meets exist and are given by
\[
\bigsqcap S \;=\; \bigsqcup\{\,x\in Q \mid \forall s\in S,\ x\sqsubseteq s\,\}.
\]
We nevertheless emphasize the join-semilattice structure since axioms
{\rm\textbf{(Q2)}--\textbf{(Q3)}} are exactly the distributivity of $\odot$ over
arbitrary joins in each argument.
\end{remark}

\medskip

\begin{definition}\label{keymodule} Given a unital quantale 
$\mathcal{Q}=(Q, \bigsqcup,\odot, e)$, a {\em left unital ${\mathcal{Q}}$-module} 
		is a triple $\mathcal{A}=(A,\bigvee,\bullet)$ such that $(A,\bigvee)$ is a complete $\bigvee$-semilattice and an {\em action} $\bullet\colon Q\times A\longrightarrow A$ is a map satisfying:
        \medskip 
	\begin{enumerate}[label=(\textbf{A\arabic*}),leftmargin=1.57905cm]
		\item $v \bullet (\bigvee S)=\bigvee_{s\in S}(v \bullet s)$ for every $S\subseteq A$ and every $v \in Q$.
		\item $(\bigsqcup T)\bullet a=\bigvee_{t\in T}(t\bullet a)$ for every $T\subseteq Q$ and every $a \in A$.
		\item $u\bullet(v\bullet a)=(u\odot v)\bullet a$ for every $u,v \in Q$ and every $a\in A$.
		\item $e \bullet a=a$ for all $a\in A$ (unitality).\\
	\end{enumerate}
Clearly, every unital quantale 
$\mathcal{Q}$ is a left ${\mathcal{Q}}$-module.
    A {\em homomorphism of left ${\mathcal{Q}}$-modules} is a 
map $g\colon A_1 \to A_2$ between ${\mathcal{Q}}$-modules $\mathcal{A}_1$ and 
$\mathcal{A}_2$
that preserves arbitrary joins and action.
		\end{definition}
Right quantale modules are defined analogously:
\medskip

\begin{definition}\label{keymodule-right}
Given a unital quantale $\mathcal{Q}=(Q,\bigsqcup,\odot,e)$, a {\em right unital ${\mathcal{Q}}$-module}
is a triple $\mathcal{A}=(A,\bigvee,\bullet)$ such that $(A,\bigvee)$ is a complete $\bigvee$-semilattice and an {\em action}
\[
\bullet \colon A\times Q\longrightarrow A
\]
is a map satisfying:
\medskip
\begin{enumerate}[label=(\textbf{A\arabic*}),leftmargin=1.57905cm]
\item $(\bigvee S)\bullet v=\bigvee_{s\in S}(s\bullet v)$ for every $S\subseteq A$ and every $v\in Q$.
\item $a\bullet(\bigsqcup T)=\bigvee_{t\in T}(a\bullet t)$ for every $T\subseteq Q$ and every $a\in A$.
\item $(a\bullet u)\bullet v=a\bullet(u\odot v)$ for every $u,v\in Q$ and every $a\in A$.
\item $a\bullet e=a$ for all $a\in A$ (unitality).\\
\end{enumerate}
Clearly, every unital quantale $\mathcal{Q}$ is a right ${\mathcal{Q}}$-module under its multiplication.
A {\em homomorphism of right ${\mathcal{Q}}$-modules} is a map $g\colon A_1\to A_2$
between right ${\mathcal{Q}}$-modules $\mathcal{A}_1$ and $\mathcal{A}_2$
that preserves arbitrary joins and the action, i.e.
\[
g\Big(\bigvee S\Big)=\bigvee_{s\in S}g(s)
\quad\text{and}\quad
g(a\bullet v)=g(a)\bullet v.
\]
\end{definition}

In the 1960s, David Foulis introduced a new mathematical structure, initially known as a ``Baer $\ast$-semigroup'' and later often termed a ``Foulis semigroup.'' This structure was fundamentally inspired by the properties of the multiplicative semigroup of bounded operators on a Hilbert space. Foulis quantales are exactly those unital involutive quantales that also possess the structural properties of Foulis semigroups.

\medskip

\begin{definition} \cite{BPS2025} \label{defFoulis}
    A {\em Foulis quantale} is a unital involutive quantale $\mathcal{Q}=(Q,\bigsqcup, \odot, ^*, e)$
together with an endomap $\sai{-} \colon Q\rightarrow Q$ satisfying 
the following conditions for all $s,x \in Q$:
\medskip
\begin{enumerate}[label=(\textbf{FQ\arabic*}),leftmargin=1.57905cm]
\item $\sai{s}$ is a self-adjoint idempotent, {i.e.},~$\sai{s} \odot \sai{s} = \sai{s} = \sai{s}^{*}$,

\item $0\, {=} \, \sai{e}$,

\item $s\odot x = 0$ if and only if there exists ${y\in Q} \text{ such that }{x = \sai{s}\odot y}$,

\end{enumerate}
\medskip
where $0=\bigsqcup \emptyset$. For an arbitrary $t\in Q$ put $t^{\perp}
  \,\smash{\stackrel{\textrm{def}}{=}}\, \sai{t^{*}} \in \sai{Q}$.
  Hence from~(\textbf{FQ1}) we get equations $t^{\perp} \cdot t^{\perp} =
  t^{\perp} = (t^{\perp})^{*}$. %
  We will call elements of $\sai{Q}$ {\em Sasaki projections}. 
  Since the unary operations $\sai{-} \colon Q\rightarrow Q$ 
  and ${}^{\perp}$ are interdefinable with $\sai{t}=(t^{*})^{\perp}$, we prefer to use the operation ${}^{\perp}$ as the key one.
  \end{definition}
\medskip
\begin{remark}
For a unital involutive quantale $\mathcal{Q}$ and $s,x \in Q$, we write $s\perp x$ if and only if $s^{*}\odot x=0$.
\end{remark}
  \medskip
Recall the following characterization of  Foulis quantales.
\medskip

\begin{proposition}\label{remFouldef} 
A unital involutive quantale $\mathcal{Q}=(Q,\bigsqcup, \odot, ^*, e)$ is a Foulis quantale with an endomap $\sai{-} \colon Q\rightarrow Q$ if and only if 
there is  an endomap $-\!^{\perp}\colon Q\to Q$ satisfying 
the following conditions for all $s,x \in Q$:
\begin{enumerate}[label={\rm (\textbf{O\arabic*})},leftmargin=1.57905cm]
    \item $s^{\perp}$ is a self-adjoint idempotent, i.e.,  $s^{\perp}\odot s^{\perp}\ =\ s^{\perp}\ =\ \big{(}s^{\perp}\big{)}^{*}$.
    \item $0\!=\!e^{\perp}$,
    \item $s \perp x$ if and only if there exists $y \in Q$ such that $x = s^{\perp}\odot y$.
\end{enumerate}
\end{proposition}

\begin{proof}
$(\Rightarrow)$ Suppose first that $\mathcal Q$ is a Foulis quantale with the map $\sai{-}: Q \to Q$ as in Definition~\ref{defFoulis}. Define the map ${-}^{\perp}: Q\to Q$ by setting
\[
s^{\perp} := \sai{s^{*}}.
\]
We verify \textup{(\textbf{O1})\,--\,(\textbf{O3})}.

\smallskip\noindent
\textup{(\textbf{O1}):} From the definition of a Foulis quantale we have
\[
s^{\perp}\odot s^{\perp}
=\sai{s^{*}}\odot\sai{s^{*}}
=\sai{s^{*}}
=s^{\perp},
\]
and similarly
\[
(s^{\perp})^{*}
=\big(\sai{s^{*}}\big)^{*}
=\sai{s^{*}}
=s^{\perp}.
\]
Hence $s^{\perp}$ is a self-adjoint idempotent.

\smallskip\noindent
\textup{(\textbf{O2}):} By the definition of a Foulis quantale we have $\sai{e}=0$. Thus
\[
e^{\perp}=\sai{e^{*}}=\sai{e}=0.
\]

\smallskip\noindent
\textup{(\textbf{O3}):} Condition (\textbf{FQ3})  in the definition of a Foulis quantale states
\[
s\odot x=0 \quad\iff\quad \exists y\colon x=\sai{s}\odot y.
\]
Replacing $s$ by $s^{*}$, this becomes
\[
s^{*}\odot x=0 \quad\iff\quad \exists y\colon x=\sai{s^{*}}\odot y.
\]
Using the definition $s^{\perp}=\sai{s^{*}}$ and the notation $s\perp x\iff s^{*}\odot x=0$, we get precisely
\[
s\perp x \quad\iff\quad \exists y\colon x=s^{\perp}\odot y.
\]

\medskip\noindent
$(\Leftarrow)$ Conversely, assume we have an involutive unital quantale $\mathcal Q$ together with an endomap ${-}^{\perp}:Q\to Q$ satisfying 
\textup{(\textbf{O1})\,--\,(\textbf{O3})} above. Define the map $\sai{-}: Q\to Q$ by
\[
\sai{s}:=(s^{*})^{\perp}.
\]
We verify (\textbf{FQ1})--(\textbf{FQ3})  from Definition~\ref{defFoulis}.

\smallskip\noindent
\textup{(\textbf{FQ1}):} By \textup{(\textbf{O1}):} we have
\[
\sai{s}\odot \sai{s}
=(s^{*})^{\perp}\odot (s^{*})^{\perp}
=(s^{*})^{\perp}
=\sai{s},
\qquad
(\sai{s})^{*}
=\big((s^{*})^{\perp}\big)^{*}
=(s^{*})^{\perp}
=\sai{s}.
\]
Thus $\sai{s}$ is a self-adjoint idempotent.

\smallskip\noindent
\textup{(\textbf{FQ2}):}  By \textup{(\textbf{O2}):} we get
\[
\sai{e}=(e^{*})^{\perp}=e^{\perp}=0.
\]

\smallskip\noindent
\textup{(\textbf{FQ3}):}  Finally, \textup{(\textbf{O3}):} says that 
\[
s\perp x\iff\exists y\colon x=s^{\perp}\odot y.
\]
Recalling $s\perp x\iff s^{*}\odot x=0$, replace $s$ by $s^{*}$ to obtain
\[
s\odot x=0\iff\exists y\colon x=(s^{*})^{\perp}\odot y=\sai{s}\odot y.
\]

Thus, the conditions (\textbf{FQ1})--(\textbf{FQ3})  are satisfied, completing the equivalence.
\end{proof}

Recall the following theorem. 

\medskip

\begin{theorem} \label{th4th5}{\rm\cite[Theorem 4, Theorem 5]{BPS2025}} Let 
$\mathcal Q$ be a Foulis quantale. Then the relation $\leq$ defined as $s\leq t$ if and only if $s=t\odot s$ for all $s, t\in {Q}$, satisfies 
$$\begin{array}{rcl}
r^{*}\odot t=0
& \Longleftrightarrow &
t=r^{\perp}\odot t \Longleftrightarrow t\leq r^{\perp},
\end{array}\eqno{(*)\phantom{**}}$$
$$\begin{array}{rclcl}
t\leq  r
& \Longrightarrow &
r^{\perp}\leq t^{\perp} &\text{and}&
k^{\perp\perp}=k,
\end{array}\eqno{(**)\phantom{*}}$$
$$\begin{array}{rcl}
t\leq  r^{\perp}
& \Longleftrightarrow &
r\leq t^{\perp} 
\end{array}\eqno{(***)}$$
for all $t, r\in Q$ and $k\in \sai{Q}$.

Moreover,  $\sai{Q}$ is a left $Q$-module   with 
order $\leq$ restricted to $\sai{Q}$, top element $e=e^{\perp\perp}=0^{\perp}$, 
join $\bigvee$ defined as 
$\bigvee S=\left(\bigsqcup S\right)^{\perp\perp}$ for all 
$S\subseteq \sai{Q}$ and 
action $\bullet$ defined as 
$u\bullet k=(u\odot k)^{\perp\perp}$ for all $u\in Q$ and $k\in \sai{Q}$.

$\sai{Q}$ is  also a right $\Cat{2}$-module where $\Cat{2}$ is the two element unital quantale with quantale module action $\circ$ defined as \[
\circ\colon \sai{Q}\times {\mathbf 2}\to \sai{Q},\quad x\circ b = \begin{cases}
    x&\text{if  } b=e,\\
    0&\text{if  } b=0.
\end{cases}
\]
\end{theorem}

\medskip
\begin{remark}
\smallskip\noindent{\rm (1)} The bracket $\sai{-}$ is an annihilator/pseudocomplement-type operation for a Foulis quantale $\mathcal Q$: for $s \in Q$, we define $[s] := (s^{*})^{\perp}$. In particular (cf. Theorem~\ref{th4th5}), for all $t \in Q$,
\[s^{*} \odot t = 0 \iff t \le s^{\perp}.
\]
Hence, $s^{\perp}$ is the greatest element annihilated by $s^{*}$ with respect to the transitive relation $\le$ (though not necessarily with respect to the order $\sqsubseteq$ on $\mathcal Q$).

\smallskip
\noindent
{\rm (2)} The operation ${}^{\perp}$ is an orthocomplementation on the ``test''  part $\sai{Q}$ but not on $Q$. The operation 
$c := {}^{\perp\perp}$ is used throughout as the
canonical projection/closure onto tests; in particular, $c(x) = x$ for all $x \in \sai{Q}$, and 
$y\leq c(y)$ for all $y\in Q$
and the module action is defined by $u \bullet k = c(u \odot k)$ 
 for all $u\in Q$ and $k\in \sai{Q}$.
\end{remark}

  \medskip

Foulis quantales and orthomodular lattices are deeply interconnected, forming a natural correspondence in the study of quantum logic and fuzzy set theory. A complete orthomodular lattice can be used to construct a Foulis quantale, and conversely, any Foulis quantale yields a complete orthomodular lattice.

\subsection{The category $\mathbb{COL}$ of complete orthomodular lattices}

Both Boolean
algebras and ortholattices have an orthocomplementation. It is distributivity that
characterizes the Boolean algebras among the ortholattices (A Boolean algebra is a distributive
ortholattice). Ortholattices as weaker structure capture mathematical and logical relationships that traditional Boolean algebras cannot express.


\medskip

\begin{definition}
An \textit{ortholattice} is a tuple $\mathcal{M} = \left( {M, \le ,\mathop  {}^ \bot} \right)$ satisfying the following conditions:
\begin{enumerate}
  \item $\left( {M, \le} \right)$ is a bounded lattice with least element $0$ and greatest element $1$;
  \item  The map ${}^ \bot  :M \to M$ satisfies, for all $m,n\in M$,
    \begin{enumerate}
    \renewcommand{\labelenumi}{\Alph{enumi}.}
    
    \item  $m\wedge m^{\perp}=0$ and $m\vee m^{\perp}=1$,
    
    \item  $m \le n \Rightarrow {n^ \bot } \le {m^ \bot }$; 

    \item  $(m^{\perp})^{\perp}=m$,
\end{enumerate}
\end{enumerate}
An ortholattice $\mathcal{M}$ is said to be an \textit{orthomodular lattice} if for all $m,n \in M$ such that $m\le n$, it holds that $ n = m \vee \bigl(m^{\perp}\wedge n\bigr)$.
\end{definition}
\medskip

We now present the definition for the category of orthomodular lattices and the corresponding ortholattice isomorphisms.

\medskip

\begin{definition}
Given two orthomodular lattices ${\mathcal{M}_1} = \left( {{M_1},{ \le _1},{}^{{ \bot _1}} } \right)$ and ${\mathcal{M}_2} = \left( {{M_2},{ \le _2},{}^{{ \bot _2}} } \right)$, an \textit{ortholattice isomorphism} $g:{\mathcal{M}_1} \to {\mathcal{M}_2}$ is a function $g:{M_1} \to {M_2}$ that satisfies the following conditions:
\begin{enumerate}
\item $g$ is a bijection,
\item ${m} \le_1 {n} \Leftrightarrow g\left( {{m}} \right) \le_2 g\left( {{n}} \right)$,
\item $g\left( {{m^{{ \bot _1}}}} \right) = {\left( {g\left( m \right)} \right)^{^{{ \bot _2}}}}$
\end{enumerate}
for all ${m},{n} \in {M_1}$. If ${\mathcal{M}_1} ={\mathcal{M}_2}$, we say that $g$ is an \textit{ortholattice automorphism}.
We denote by $\mathbb{COL}$ the category of complete orthomodular lattices and ortholattice isomorphisms.
\end{definition}

\medskip

\begin{remark}
Let $g:{\mathcal{M}_1} \to {\mathcal{M}_2}$ be an ortholattice isomorphism. Then for all \(m,n\in M_1\), we have $g(m \vee n) = g(m)\vee g(n)$ and $ g(m\wedge n) = g(m) \wedge g(n) $.
\end{remark}
\medskip

Additionally, for an ortholattice $\mathcal{M} = (M, \le, {}^{\perp})$, we define two key operations for each element $m \in M$:
\begin{enumerate}
    \item \textit{Sasaki projection} (onto $m$) : This map, denoted ${\pi _m}\colon M \to M$, takes an element $n$ to $m \land (m^\perp \lor n)$,

    \item \textit{Sasaki hook} (from $m$) : This map, denoted ${\pi^m}\colon M \to M$, takes an element $n$ to $m^\perp \lor (m \land n)$.
\end{enumerate}
\medskip
These two maps are always order-preserving. A significant property is that the ortholattice $\mathcal{M}$ is orthomodular if and only if, for every $m \in M$, the Sasaki projection $\pi _m$ is the left order adjoint of the Sasaki hook $\pi^m$. This means that in an orthomodular lattice, Sasaki projections preserve arbitrary existing joins.
\medskip

\begin{definition}\label{linearmap1}
A function $f: X \to Y$ is defined as a \textit{linear map} (or \textit{adjoint map}) from an orthomodular lattice $\mathcal{X}$ to an orthomodular lattice $\mathcal{Y}$ if there exists a function $g: Y \to X$, denoted  $f^*$ and called the \textit{adjoint} of $f$ (uniqueness follows from \cite[Lemma 1]{BPS2025}), such that for all $x \in X$ and  $y \in Y$, the following equivalence holds:
$$f(x) \perp_Y y \quad \iff \quad x \perp_X g(y),$$
where $a \perp_{\mathcal M} b$ in an orthomodular lattice $\mathcal M=(M,\le_{\mathcal M},{}{}^{\perp})$ for $a,b\in M$ means $a \le_{\mathcal M} b^{\perp}$.

The collection of all linear maps from an orthomodular lattice $\mathcal{X}$ to an orthomodular lattice $\mathcal{Y}$ is denoted by $\textbf{Lin}(\mathcal{X},\mathcal{Y})$. We write $\textbf{Lin}(\mathcal{X})=\textbf{Lin}(\mathcal{X},\mathcal{X})$. Moreover, every ortholattice isomorphism 
$k\colon \mathcal{X}\to \mathcal{Y}$ has an adjoint $k^{-1}$. Note that for any orthomodular lattice $\mathcal X$, $\pi_x \in \textbf{Lin}(\mathcal X)$ such that 
$\pi_x=(\pi_x)^{*}$ for all $x \in X$, and 
every ortholattice automorphism is in $\textbf{Lin}(\mathcal{X})$.

Furthermore, if both $\mathcal{X}$ and $\mathcal{Y}$ are complete orthomodular lattices, then $\textbf{Lin}(\mathcal{X},\mathcal{Y})$ forms a complete lattice and $\textbf{Lin}(\mathcal{X})$ is a Foulis quantale \cite[Proposition 1]{BPS2025}. 
\end{definition}

\medskip

This definition is central in the study of orthomodular lattices from a categorical perspective, where such maps serve as morphisms. The existence of the adjoint function $g$ captures a specific type of structural preservation.

\medskip

The following theorem characterizes Sasaki projections. 
\medskip
\begin{theorem}[Characterization of Sasaki projections]\label{charsas}{\rm \cite{Kalmbach83}}
Let $\mathcal M$ be an orthomodular lattice. A map $\phi: M \to M$ is a Sasaki projection $\pi_m$ for some element $m \in M$ if and only if it satisfies the following properties:
\begin{enumerate}
    \item \textbf{Idempotence:} $\phi^2 = \phi$, meaning $\phi(\phi(x)) = \phi(x)$ for all $x \in M$.
    \item \textbf{Image:} $\text{im } \phi = \downarrow m$, where $\downarrow m = \{x \in M \mid x \leq m\}$ and $\phi(1)=m$.
    \item \textbf{Self-adjointness:} $\phi$ is a self-adjoint linear map.
\end{enumerate}
\end{theorem}

\medskip
In the context of a Hilbert space $\mathcal{H}$, these mappings correspond uniquely to the action of self-adjoint linear idempotent operators (orthogonal projections) on the lattice of closed subspaces.
\medskip

\medskip\noindent

\begin{remark}\label{rem:two-adjoints-sasaki}
There are \emph{two different notions of adjoint} in use, and they should not be conflated.

\smallskip
\noindent
\textbf{(1) Order adjoint (Galois adjunction):}
For monotone maps $f,g\colon M\to M$, we write $f\dashv g$ if
\[
f(x)\le y \quad\Longleftrightarrow\quad x\le g(y)
\qquad(x,y\in M).
\]
In an orthomodular lattice $\mathcal M$ and $m\in M$, the Sasaki projection
\[
\pi_m(x)=m\wedge(m^\perp\vee x)
\]
is a left order adjoint to the Sasaki hook
\[
(m\Rightarrow -)\colon M\to M,\qquad
m\Rightarrow y:=m^\perp\vee(m\wedge y),
\]
that is,
\[
\pi_m(x)\le y \quad\Longleftrightarrow\quad x\le (m\Rightarrow y)
\qquad(x,y\in M).
\]
(This is noted in the paragraph preceding Definition~\ref{linearmap1}.)

\smallskip
\noindent
\textbf{(2) Orthogonality adjoint (the involution in $\mathbf{Lin}(\mathcal M)$):}
In the Foulis quantale $\mathbf{Lin}(\mathcal M)$, the involution ${}^{*}$ is \emph{not} the order adjoint.
Rather, for a join-preserving map $f\colon M\to M$ the map $f^{*}\colon M\to M$ is characterized by
\[
f(x)\perp y \quad\Longleftrightarrow\quad x\perp f^{*}(y),
\qquad(x,y\in M),
\]
where $u\perp v$ means $u\le v^\perp$.
Equivalently, if $f_\dashv$ denotes the (right) order adjoint of $f$ in the sense of (1), then
\[
f^{*}(y)=\bigl(f_\dashv(y^\perp)\bigr)^\perp.
\]

\smallskip
\noindent
In particular, for the Sasaki projection $\pi_m$ we have $\pi_m\dashv(m\Rightarrow -)$, hence
\[
(\pi_m)^{*}(y)=\bigl((m\Rightarrow y^\perp)\bigr)^\perp.
\]
A direct computation shows that $(\pi_m)^{*}=\pi_m$:
\[
\begin{aligned}
(\pi_m)^{*}(y)
&=\bigl(m^\perp\vee(m\wedge y^\perp)\bigr)^\perp %
=(m^\perp)^\perp\ \wedge\ (m\wedge y^\perp)^\perp \\
&=m\ \wedge\ (m^\perp\vee y) %
=\pi_m(y).
\end{aligned}
\]
Thus the statement ``$\pi_x=(\pi_x)^{*}$'' (after Definition~\ref{linearmap1}) means that Sasaki projections are
\emph{self-adjoint with respect to the orthogonality adjoint ${}^{*}$ in $\mathbf{Lin}(\mathcal M)$},
and it does \emph{not} contradict the order-adjunction $\pi_m\dashv(m\Rightarrow -)$ from (1).
\end{remark}

\medskip\noindent
Now assume that $\mathcal X$ and $\mathcal Y$ are \emph{complete} orthomodular lattices.
We equip $\mathbf{Lin}(\mathcal X,\mathcal Y)$ with the \emph{pointwise order}
\[
f \sqsubseteq g
\quad\Longleftrightarrow\quad
(\forall x\in X)\; f(x)\le_{\mathcal Y} g(x).
\]
For any family $\{f_i\}_{i\in I}\subseteq \operatorname{Lin}(\mathcal X,\mathcal Y)$ we define
joins pointwise by
\[
\Big(\bigsqcup_{i\in I} f_i\Big)(x)\;:=\;\bigvee_{i\in I} f_i(x),
\qquad (x\in X).
\]
The bottom element is the zero map $0_{\mathcal X,\mathcal Y}:X\to Y$ defined by $0_{\mathcal X,\mathcal Y}(x)=0_{\mathcal Y}$.

\medskip\noindent
When $\mathcal X=\mathcal Y$, we write $\operatorname{Lin}(\mathcal X):=\operatorname{Lin}(\mathcal X,\mathcal X)$
and define the quantale multiplication and unit by
\[
f\odot g \;:=\; f\circ g,
\qquad
e \;:=\; \operatorname{id}_X,
\]
and the involution by the adjoint $f^{*}$.
Moreover, writing $\pi_a(y):=a\wedge(a^{\perp}\vee y)$ for the Sasaki projection in $\mathcal X$,
the Foulis bracket is given by
\[
[f]\;:=\;\pi_{\,f^{*}(1)^{\perp}}
\qquad(f\in \operatorname{Lin}(\mathcal X)),
\]
so that $\operatorname{Lin}(\mathcal X)$ becomes a Foulis quantale \cite[Proposition~1]{BPS2025}.
(Equivalently, one may set $f^{\perp}:=[f^{*}]=\pi_{\,f(1)^{\perp}}$).
For $x\in X$ we have 
\[
[\pi_{x}]\;=\;\pi_{\,\pi_{x}^{*}(1)^{\perp}}%
\;=\;\;\pi_{\,\pi_{x}(1)^{\perp}}\;=\;\;\pi_{x^{\perp}}.
\]

\medskip

\begin{example}[Hilbert lattice witness for non-monotonicity of Sasaki projections]
\label{ex:hilbert-sasaki-nonmonotone}
{\rm Let $H=\mathbb{R}^3$ with the standard inner product, and let $\mathcal X$ be the lattice of
all linear subspaces of $H$, ordered by inclusion. Then $\mathcal X$ is an orthomodular
lattice with
\[
A\wedge B = A\cap B,\qquad
A\vee B = \operatorname{span}(A\cup B),\qquad
A^\perp=\{h\in H:\langle h,a\rangle=0\ \forall a\in A\}.
\]
For $u\in X$, define the Sasaki projection
\[
\pi_u:X\to X,\qquad
\pi_u(x):=u\wedge (x\vee u^\perp)=u\cap \operatorname{span}(x\cup u^\perp).
\]

Put
\[
u:=\operatorname{span}(e_1),\qquad
v:=\operatorname{span}(e_1,e_2),
\]
so $u\le v$. Let
\[
x:=\operatorname{span}(e_1+e_2+e_3).
\]
We compute $\pi_u(x)$ and $\pi_v(x)$.

First, $u^\perp=\operatorname{span}(e_2,e_3)$. We claim that
\[
x\vee u^\perp=\operatorname{span}(x\cup u^\perp)=\mathbb{R}^3.
\]
Indeed, $e_2,e_3\in u^\perp\subseteq \operatorname{span}(x\cup u^\perp)$ and
$e_1=(e_1+e_2+e_3)-e_2-e_3\in \operatorname{span}(x\cup u^\perp)$, hence
$e_1,e_2,e_3\in \operatorname{span}(x\cup u^\perp)$, so
$\operatorname{span}(x\cup u^\perp)=\mathbb{R}^3$. Therefore
\[
\pi_u(x)=u \wedge \operatorname{span}(x\cup u^\perp)=u\wedge \mathbb{R}^3=u=\operatorname{span}(e_1).
\]

Next, $v^\perp=\operatorname{span}(e_3)$. We claim that
\[
x\vee v^\perp=\operatorname{span}(x\cup v^\perp)=\operatorname{span}(e_1+e_2,e_3).
\]
Indeed, $e_3\in v^\perp\subseteq \operatorname{span}(x\cup v^\perp)$ and
$e_1+e_2=(e_1+e_2+e_3)-e_3\in \operatorname{span}(x\cup v^\perp)$, so
$\operatorname{span}(e_1+e_2,e_3)\subseteq \operatorname{span}(x\cup v^\perp)$.
Conversely, $e_1+e_2+e_3=(e_1+e_2)+e_3\in \operatorname{span}(e_1+e_2,e_3)$, hence
$x\subseteq \operatorname{span}(e_1+e_2,e_3)$ and also
$v^\perp\subseteq \operatorname{span}(e_1+e_2,e_3)$, yielding
$\operatorname{span}(x\cup v^\perp)\subseteq \operatorname{span}(e_1+e_2,e_3)$.
Thus equality holds. Therefore
\[
\pi_v(x)
=
v\wedge \operatorname{span}(e_1+e_2,e_3)
=
\operatorname{span}(e_1,e_2)\cap \operatorname{span}(e_1+e_2,e_3)
=
\operatorname{span}(e_1+e_2).
\]

Now $\operatorname{span}(e_1)\not\subseteq \operatorname{span}(e_1+e_2)$, since
$e_1=\lambda(e_1+e_2)$ would force $0=\lambda$ from the $e_2$-coordinate, hence
$e_1=0$, a contradiction. Therefore $\pi_u(x)\nleq \pi_v(x)$.
Consequently, in the pointwise order on maps $X\to X$, we have 
\[
u\le v
\qquad\text{but}\qquad
\pi_u \nleq \pi_v,
\]
as witnessed by the element $x\in X$.}
\end{example}

\medskip

\begin{remark}\label{rem:sasaki-proj-terminology}
The term \emph{Sasaki projection} is used in two closely related senses in this paper.

\smallskip
(1) If $\mathcal Q$ is an abstract Foulis quantale, then by definition we call the elements of
\(
[Q]:=\{\,t^{\perp}\mid t\in Q\,\}
\)
\emph{Sasaki projections}.

\smallskip
(2) If $\mathcal M$ is a complete orthomodular lattice, then we also consider the standard \emph{Sasaki projection maps}
\(
\pi_m\colon M\to M,\quad
\pi_m(b):=m\wedge(m^{\perp}\vee b)
\)
(and the associated Sasaki hooks).

\smallskip
These notions agree in the basic example 
$\mathcal Q=\mathbf{Lin}(\mathcal M)$. 
Namely, by Theorem \ref{charsas} and {\rm \cite[Proposition 26]{botur2025foulis}} we have that 
\[
[\,\mathbf{Lin}(\mathcal M)\,]=\{\,\pi_m \mid m\in M\,\}.
\]
\end{remark}

\medskip

\begin{example}\label{Sasaki is Involutive}{\rm 
 Let $\mathcal{M} = \left( {M, \le ,{}^ \bot} \right)$ be a complete orthomodular lattice. Define the set $\textbf{Si}({M})$ as the set of all finite compositions of Sasaki projections on $\mathcal{M}$:
$$\textbf{Si}({M}) = \left\{ {{\pi _{{m_1}}} \circ \cdots \circ {\pi _{{m_j}}} \mid {m_1}, \ldots ,{m_j} \in M, j \in {\mathbb{N}^+}} \right\}$$
This structure $\textbf{Si}(\mathcal{M})=\left(\textbf{Si}(M),\circ, \operatorname{id}_M,{{}^ *}\right)$ is the smallest involutive submonoid 
of $\textbf{Lin}(\mathcal{M})$ containing all Sasaki projections $\{\pi_m : M\to M \mid m\in M\}$.}
\end{example}

\medskip

The following results establish the correspondence between Foulis quantales and complete orthomodular lattices.

\medskip

\begin{theorem}
\label{FoulisOMKerLem} {\rm\cite[Theorem 4]{BPS2025}}
Let $\mathcal Q$ be a Foulis quantale. Then the structure $\sai{\mathcal Q}=(\sai{Q}, \leq, {}^{\perp})$ 
  is a complete orthomodular lattice where 

$$\begin{array}{lrcl}
&\sai{Q}& = &
\set{\sai{t}}{t\in Q}  \subseteq Q\\
\mbox{Order} & k_{1}\leq k_{2} & \Leftrightarrow & k_{1} = k_{2}\odot k_{1} \\
\mbox{Top} & e & = &  
   \sai{0} \\
\mbox{Orthocomplement\quad\quad} & k^{\perp} & = & \sai{k} \\
\mbox{Meet} & k_{1} \conjun k_{2} & = &  
   \big(k_{1} \odot \sai{\sai{k_{2}}\odot k_{1}}\big)^{\perp\perp}\\
\mbox{Join} & \bigvee S & = &  \sai{\sai{\bigsqcup S}}.
\end{array}$$
\end{theorem}

\medskip

\begin{lemma}\label{lem:perp-0-e}
In every Foulis quantale 
$\mathcal{Q}=(Q, \bigsqcup,\odot, {}^{*}, {}^{\perp}, e)$, the following holds:

 \medskip

\begin{enumerate}
    \item $0^{\perp}=e=e^{\perp\perp}$, $e^{\perp}=0=0^{\perp\perp}$,
    \item $(x\odot y^{\perp\perp})^{\perp}=(x\odot y)^{\perp}$ for all $x, y\in Q$,
    \item $\Big(\bigsqcup_{i \in I}  x_{i}^{\perp\perp}\Big)^{\perp}
\;=\;
\Big(\bigsqcup_{i \in I}  x_{i}\Big)^{\perp}$ 
for every  indexed family $\{x_i\}_{i\in I}\subseteq Q$, 
\item $(x^{\perp\perp}\odot y)^{\perp\perp}=\left(x^{\perp}\sqcup (x^{\perp}\sqcup y)^{\perp}\right)^{\perp}$ for all $x, y\in Q$.
\end{enumerate}
\end{lemma}

\begin{proof}1. Since $0, e\in \sai{Q}$ we conclude 
from (**) in Theorem~\ref{th4th5} that $0=0^{\perp\perp}$ 
and $e=e^{\perp\perp}$. The remaining part is evident.

\medskip
\noindent{}2. Since $\leq$ is an order on $\sai{Q}$ we have by repeated application of (*) in Theorem~\ref{th4th5} that, 
for all $x, y\in Q$ and all $t\in \sai{Q}$,
\begin{align*}
    t\leq (x\odot y)^{\perp} &\Longleftrightarrow %
    (x\odot y)^{*}\odot t=0 \Longleftrightarrow %
    y^{*}\odot x^{*}\odot t=0\Longleftrightarrow %
    x^{*}\odot t\leq y^{\perp}=y^{\perp\perp\perp}\\
    &\Longleftrightarrow %
\left(y^{\perp\perp}\right)^{*}\odot x^{*}\odot t=0
\Longleftrightarrow %
(x\odot y^{\perp\perp})^{*}\odot t=0
\Longleftrightarrow %
t\leq (x\odot y^{\perp\perp})^{\perp}.
\end{align*}

\noindent{}3. Again, by repeated application of (*) in Theorem~\ref{th4th5} we conclude that, 
for all  indexed families $\{x_i\}_{i\in I}\subseteq Q$ and all $t\in \sai{Q}$,
\begin{align*}
t\leq \Big(\bigsqcup_{i \in I}  x_{i}^{\perp\perp}\Big)^{\perp}&\Longleftrightarrow %
\Big(\bigsqcup_{i \in I}  x_{i}^{\perp\perp}\Big)^{*}\odot t=0
\Longleftrightarrow %
\left(x_{i}^{\perp\perp}\right)^{*}\odot t=0\text{ for all }i\in I\\
&\Longleftrightarrow %
t\leq x_{i}^{\perp\perp\perp}=%
x_{i}^{\perp}\text{ for all }i\in I%
\Longleftrightarrow %
x_{i}^{*}\odot t=0\text{ for all }i\in I\\
&\Longleftrightarrow %
\Big(\bigsqcup_{i \in I}  x_{i}\Big)^{*}\odot t=0
\Longleftrightarrow %
t\leq \Big(\bigsqcup_{i \in I}  x_{i}\Big)^{\perp}.
\end{align*}

\medskip
\noindent{}4. Let $x\in Q$. Then by part 2, 
$(x^{\perp\perp}\odot y)^{\perp\perp}=%
(x^{\perp\perp}\odot y^{\perp\perp})^{\perp\perp}=%
x^{\perp\perp}\bullet y^{\perp\perp}$ 
for all $y\in Q$. The map 
$x^{\perp\perp}\bullet (-)\colon \sai{Q}\to \sai{Q}$ is 
by \cite[Proposition 26]{botur2025foulis}  self-adjoint linear, idempotent
and $\text{im } x^{\perp\perp}\bullet (-) = \downarrow x^{\perp\perp}=\downarrow x^{\perp\perp}\bullet e$. 
From Theorem \ref{charsas} we conclude  $x^{\perp\perp}\bullet (-)=\pi_{x^{\perp\perp}}$.
We compute:  
\begin{align*}
    x^{\perp\perp}\bullet y^{\perp\perp}&=%
\pi_{x^{\perp\perp}}(y^{\perp\perp})=%
x^{\perp\perp}\wedge (x^{\perp}\vee y^{\perp\perp})
=%
\left(x^{\perp}\sqcup (x^{\perp}\vee y^{\perp\perp})^{\perp}%
\right)^{\perp}\\
&=%
\left(x^{\perp}\sqcup (x^{\perp}\sqcup y^{\perp\perp})^{\perp}%
\right)^{\perp}=%
\left(x^{\perp}\sqcup (x^{\perp}\sqcup y)^{\perp}%
\right)^{\perp}.
\end{align*}

\end{proof}

\medskip
\begin{proposition}\label{rqinv}\cite[Proposition 2]{BPS2025}
	Let $X$ be a complete orthomodular lattice. Then 
	$X$ is a left $\Cat{Lin}(X)$-module and also a right $\Cat{2}$-module, 
    where $\Cat{2}$ is the {\bfseries two-element} unital quantale.
\end{proposition}


\section{Involutive generalized dynamic algebras}\label{igedyal}

This section introduces the concept of an involutive generalized dynamic algebra, which serves as a foundational algebraic structure motivated by the properties found in Foulis quantales and function-based orthomodular dynamic algebras. This framework extends the properties of involutive unital quantales by incorporating an additional unary operation, ${\sim}$, which captures a form of test or projection relevant to non-classical logic and quantum computing models.


        \medskip
For a comprehensive background on generalized dynamic algebras, we refer the reader to \cite{KRSZ,Rad2025}.

\medskip

\begin{definition}\label{def:igda}
An \emph{involutive generalized dynamic algebra} is a tuple \(\mathfrak{K} = (K, \bigsqcup, \odot, {}^{*}, {\sim}, e)\) satisfying the following conditions:
\medskip 
\begin{enumerate}[label=(\textbf{IDA\arabic*}),leftmargin=1.57905cm]
    \item \((K, \bigsqcup, \odot, {}^{*}, e)\) is an involutive unital quantale.
    \item[]\({\sim} \colon K \to K\) is a  unary operation on \(K\) such that 
    \item ${\sim}(x\odot {\sim}\, {\sim} y)=%
    {\sim}(x\odot y)$ for all $x, y\in K$,
    \item ${\sim}\Big(\bigsqcup_{i \in I}  \,{\sim}\, {\sim} x_{i}\Big)
\;=\;
{\sim}\Big(\bigsqcup_{i \in I}  x_{i}\Big)$ 
for all indexed families $\{x_i\}_{i\in I}\subseteq K$,
   \item   $({\sim} x)^{*}={\sim} x$ for all $x\in K$, 
   \item ${\sim}\, {\sim} ({\sim}\, {\sim} x\odot y)=%
   {\sim}\left({\sim} x\sqcup {\sim}({\sim}x\sqcup y)\right)$ for all $x, y\in K$.
\end{enumerate}
\end{definition}
\medskip
A Foulis quantale is an involutive generalized dynamic algebra by Lemma~\ref{lem:perp-0-e}, and every involutive generalized dynamic algebra is, in turn, a generalized dynamic algebra in the sense defined by \cite{KRSZ}.

\medskip
We now present several constructions involving an involutive generalized dynamic algebra \(\mathfrak{K} = (K, \bigsqcup, \odot, {}^{*}, {\sim}, e)\) .
\begin{equation*}
\begin{split}
\widetilde K & \stackrel{\text{def}}{=} \{ {\sim} k \mid k \in K \} \\
\bigvee W & \stackrel{\text{def}}{=} \,\, {\sim} \left({\sim} \bigsqcup W\right), \, \text{for any}\, W \subseteq \widetilde K\\
\bigwedge W & \stackrel{\text{def}}{=} \,\, {\sim} \bigsqcup \left\{ { {\sim} w:w \in W} \right\}
, \, \text{for any}\, W \subseteq \widetilde K\\
w^{\perp} & \stackrel{\text{def}}{=} \,\, {\sim} w, %
\, \text{for any}\, w \in \widetilde K\\
\preceq \,\,\,\,\, & \stackrel{\text{def}}{=} \,\, \left\{ {\left( {k,l} \right) \in \widetilde K \times \widetilde K \mid \bigvee \left\{ {k,l} \right\} = l} \right\}.\\
\end{split}
\end{equation*}

For a fixed element ${k \in K}$, we define a unary operation ${\ulcorner k \urcorner : K \to K}$. This operation is formally given by:
$$ \ulcorner k \urcorner \left( l \right) =\,\, {\sim} \left( { {\sim} \left( {k \odot l} \right)} \right) $$
for every ${l \in K}$. Building on this, we introduce an equivalence relation, denoted by ${\equiv}$:
$$ \equiv \,\,\stackrel{\text{def}}{=} \,\, \left\{ {\left( {k,l} \right) \in K \times K \mid \ulcorner k \urcorner \left( w \right) = \ulcorner l \urcorner \left( w \right), \, \text{for every}\, w \in \widetilde K} \right\}. $$
This relation holds if two elements $k$ and $l$ produce the same result when their corresponding unary operations are applied to any ${w \in \widetilde K}$.

\medskip

\begin{definition}[Morphisms of involutive  generalized dynamic algebras]
Let  \(\mathfrak{K}_1 = (K_1, \bigsqcup_1, \odot_1, {}^{*}_1, {\sim_1}, e_1)\) and 
\(\mathfrak{K}_2 = (K_2, \bigsqcup_2, \odot_2, {}^{*}_2, {\sim_2}, e_2)\)
 be involutive generalized dynamic algebras.
A \emph{morphism} 
 $f:\mathfrak{K}_1\to\mathfrak{K}_2$
of involutive  generalized dynamic algebras is a
homomorphism of unital involutive quantales such that 
$f({\sim_1} x)\;=\;{\sim_2} f(x)$ for all $x\in K_1$. 
Involutive generalized dynamic algebras and their morphisms obviously form a category denoted by $\mathbb{IDA}$.
\end{definition}

\medskip

Now, we make precise how the dynamics encoded in the unital involutive quantale $K$ acts on the test set $\widetilde K$. 
We introduce the canonical left action “apply then reflect to tests,” given by $(k,v)\mapsto {\sim}{\sim}(k\odot v)$, and verify that it equips $\widetilde K$ with the structure of a $K$-module.
\medskip

\begin{theorem}\label{thm:K-module}
Let \(\mathfrak{K} = (K, \bigsqcup, \odot, {}^{*}, {\sim}, e)\) 
 be an involutive  generalized dynamic algebra.  We define an action 
$\bullet : K\times \widetilde{K}\to \widetilde{K}$ such that
\[
k\bullet v = \ulcorner k \urcorner(v) = {\sim}\,{\sim} (k\odot v)
\]
for every $k\in K$ and $v\in \widetilde{K}$. 
Then $(\widetilde{K},\bigvee,\bullet)$ is a left $K$–module with 
the least element ${\sim}\,{\sim} 0$ and the greatest element 
${\sim}\,{\sim} 1$
such that 
${\sim}\,{\sim}\,{\sim} v={\sim} v$ for every $v\in {K}$.  
\end{theorem}
\begin{proof}  Let $v\in K$. We compute:
\begin{align*}
{\sim}\, {\sim}\, {\sim} v=
    {\sim}(e\odot {\sim}\, {\sim} v)\stackrel{(\textbf{IDA2})}{=}%
    {\sim}(e\odot v)={\sim} v.
\end{align*}    
Now, let us show that 
$(\widetilde{K},\bigvee)$ is a complete lattice. Clearly, 
${\sim}\,{\sim} 0, {\sim}\,{\sim} 1\in \widetilde{K}$ and $(\widetilde{K},\preceq)$ is a bounded ordered set. Namely, reflexivity and antisymmetry are evident. 
Let $x, y, z\in \widetilde{K}$ such that $x\preceq y$ and $y\preceq z$. Then 
$\bigvee \left\{ {x,y} \right\} = y$ and $\bigvee \left\{ {y,z} \right\} = z$. We compute: 
\begin{align*}
  \bigvee \left\{ {x,z} \right\} &=  \bigvee \left\{ {x, \bigvee \left\{ {y,z} \right\}} \right\}=%
  \bigvee \left\{ {x, {\sim}\,{\sim} \bigsqcup \left\{ {y,z} \right\}} \right\}%
  \stackrel{(\textbf{IDA3})}{=}%
  \bigvee \left\{ {x,  \bigsqcup \left\{ {y,z} \right\}} \right\}\\
  &={\sim}\,{\sim} \bigsqcup \left\{ {x,  \bigsqcup \left\{ {y,z} \right\}} \right\}=%
  {\sim}\,{\sim} \bigsqcup \left\{ {x,  {y,z}} \right\}=%
  {\sim}\,{\sim} \bigsqcup \left\{ {\bigsqcup \left\{ {x,y} \right\}}, z \right\}\\%
  &\stackrel{(\textbf{IDA3})}{=}{\sim}\,{\sim} \bigsqcup  \left\{{\sim}\,{\sim} {\bigsqcup \left\{ {x,y} \right\}}, z \right\}=%
  {\sim}\,{\sim} \bigsqcup \left\{ y, z \right\}=\bigvee \left\{ {y,z} \right\} = z\\%
\end{align*}
and, for every $w\in \widetilde{K}$,
\begin{align*}
\bigvee \left\{ {w,{\sim}\,{\sim} 0} \right\} &\stackrel{(\textbf{IDA3})}{=}%
    {\sim}\,{\sim}\, {\bigsqcup \left\{ {w,0} \right\}}= {\sim}\,{\sim}\, w=w,\\
    \bigvee \left\{ {w,{\sim}\,{\sim} 1} \right\} &\stackrel{(\textbf{IDA3})}{=}%
    {\sim}\,{\sim}\, {\bigsqcup \left\{ {w,1} \right\}}= {\sim}\,{\sim}\, 1.
\end{align*}


Assume that $S\subseteq \widetilde{K}$. Let us show that 
$v={\sim}\,{\sim}\Big(\bigsqcup S\Big)$ is the join of $S$ in $\widetilde{K}$. 

Let $s\in S$. We compute:
\begin{align*}
    \bigvee \left\{ {s,v} \right\} &=  \bigvee \left\{ {s, \bigvee S} \right\}=%
  \bigvee \left\{ {s, {\sim}\,{\sim} \bigsqcup S} \right\}=%
  {\sim}\,{\sim} \bigsqcup \left\{ {s,  \bigsqcup S} \right\}= %
  {\sim}\,{\sim} \bigsqcup S=v.
\end{align*}
 Hence $v$    is an upper bound of $S$ in $\widetilde{K}$. Now, let 
 $w$    be an upper bound of $S$ in $\widetilde{K}$. Then 
 $\bigvee \left\{ {s,w} \right\} = w$ for all $s\in S$. We compute: 
 \begin{align*}
    \bigvee \left\{ {w,v} \right\} &=  \bigvee \left\{ {w, \bigvee S} \right\}=%
  \bigvee \left\{ {w, {\sim}\,{\sim} \bigsqcup S} \right\}%
  \stackrel{(\textbf{IDA3})}{=}%
  {\sim}\,{\sim} \bigsqcup \left\{ {w,  \bigsqcup S} \right\}\\
  &= {\sim}\,{\sim} \bigsqcup \left\{ {w  \sqcup s} \mid s\in S \right\}%
  \stackrel{(\textbf{IDA3})}{=}
  {\sim}\,{\sim} \bigsqcup \left\{ {\sim}\,{\sim} ({w  \sqcup s}) \mid s\in S \right\}=w.
\end{align*}
We conclude that $v$ is the least upper bound of $S$ in $\widetilde{K}$.

The four left module axioms are verified below using iterative application of the definitions for $\bigvee$ and $\bullet$.

\medskip

\noindent{}(\textbf{A1}): Let $S\subseteq \widetilde{K}$ and $v\in K$. 
We compute:
\begin{align*}
v\bullet \left(\bigvee S\right) &= {\sim}\,{\sim} \left(v\odot \left({\sim}\,{\sim} \left(\bigsqcup S\right)\right)\right)%
\stackrel{(\textbf{IDA2})}{=} {\sim}\,{\sim} \left(v\odot \left(\bigsqcup S\right)\right)%
\\
&= {\sim}\,{\sim} \left(\bigsqcup\{ v\odot s\mid s\in S\}\right)%
\stackrel{(\textbf{IDA3})}{=}{\sim}\,{\sim} \left(\left(\bigsqcup\{ v\bullet s\mid s\in S\}\right)\right)= \bigvee_{s\in S}(v\bullet s).
\end{align*}

\medskip

\noindent{}(\textbf{A2}): Let $T\subseteq {K}$ and $a\in \widetilde K$. 
We compute:
\begin{align*}
\left(\bigsqcup T\right)&\bullet a = {\sim}\,{\sim} \left(\left(\bigsqcup T\right) \odot a\right)%
= {\sim}\,{\sim} \left(\left(\bigsqcup \{t\odot a\mid t\in T\}\right) \right)
\\
&\stackrel{(\textbf{IDA3})}{=} {\sim}\,{\sim} \left(\left(\bigsqcup \{t\bullet a\mid t\in T\}\right) \right)
= \bigvee_{t\in T}(t\bullet a).
\end{align*}

\medskip

\noindent{}(\textbf{A3}): Let $u,v\in K$ and $a\in \widetilde K$. 
We compute:
\begin{align*}
u\bullet(v\bullet a)&=u\bullet({\sim}\,{\sim}(v\odot a))
={\sim}\,{\sim}\big(u\odot {\sim}\,{\sim}(v\odot a)\big)%
\stackrel{(\textbf{IDA2})}{=}{\sim}\,{\sim}\big(u\odot (v\odot a)\big)\\
&{=}{\sim}\,{\sim}\big((u\odot v)\odot a\big)=%
(u\odot v)\bullet a.
\end{align*}

\noindent{}(\textbf{A4}): Let $a\in \widetilde K$. We compute:
\begin{align*}
e\bullet a={\sim}\,{\sim}(e\odot a)={\sim}\,{\sim} a\stackrel{(\textbf{IDA2})}{=}a.
\end{align*}
\end{proof}

The following proposition establishes key structural properties concerning the  map ${\sim}\,{\sim}$ and the associated equivalence relation $\equiv$ within an involutive generalized dynamic algebra, specifically demonstrating its module and congruence properties.

\medskip

\begin{proposition}\label{prop:K-module}
Let \(\mathfrak{K} = (K, \bigsqcup, \odot, {}^{*}, {\sim}, e)\) 
 be an involutive  generalized dynamic algebra. 
Then ${\sim}\,{\sim} \colon K \to \widetilde{K}$ is a surjective homomorphism of left $K$–modules and $ \equiv$ is a quantale congruence. 
\end{proposition}
\begin{proof} Let $S\subseteq K$. We compute:
$$
{\sim}\,{\sim}\left(\bigsqcup S\right) \stackrel{(\textbf{IDA3})}{=}%
{\sim}\,{\sim}\left(\bigsqcup \{{\sim}\,{\sim} s\mid s\in S\}\right)=\bigvee_{s\in S} {\sim}\,{\sim} \left(s\right).%
$$
Similarly, let $u, v\in K$. We compute:
$$
{\sim}\,{\sim}\left(u\odot v\right)%
\stackrel{(\textbf{IDA2})}{=}%
{\sim}\,{\sim}\left(u\odot {\sim}\,{\sim} v\right)=%
u\bullet {\sim}\,{\sim} \left(v\right).
$$

\medskip
Let $u\equiv v$, $s\equiv t$ and $w\in \widetilde{K}$. 
We compute:
\begin{align*}
    (u\odot s)\bullet w&=%
    {\sim}\,{\sim} \left((u\odot s)\odot w\right)=%
     {\sim}\,{\sim} \left(u\odot (s\odot w)\right)%
     \stackrel{(\textbf{IDA2})}{=}%
{\sim}\,{\sim} \left(u\odot {\sim}\,{\sim} (s\odot w)\right)\\%
&={\sim}\,{\sim} \left(u\odot {\sim}\,{\sim} (t\odot w)\right)=%
{\sim}\,{\sim} \left(v\odot {\sim}\,{\sim} (t\odot w)\right)%
\stackrel{(\textbf{IDA2})}{=}%
{\sim}\,{\sim} \left(v\odot  (t\odot w)\right)\\
&={\sim}\,{\sim} \left((v\odot  t)\odot w\right)%
=(v\odot t)\bullet w.
\end{align*}
Hence $u\odot s\equiv v\odot t$. Assume now that $x_i\equiv y_i$, where 
$x_i, y_i\in K$, $i\in I$ and $w\in  \widetilde{K}$. Then
\begin{align*}
{\sim}\,{\sim} \left(\left(\bigsqcup_{i\in I} x_i\right)\odot w\right)&=%
{\sim}\,{\sim} \left(\bigsqcup_{i\in I} \left(x_i\odot w)\right)\right)\stackrel{(\textbf{IDA3})}{=}%
{\sim}\,{\sim} \left(\bigsqcup_{i\in I} {\sim}\,{\sim}\left(x_i\odot w)\right)\right)\\
&=%
{\sim}\,{\sim} \left(\bigsqcup_{i\in I} {\sim}\,{\sim}\left(y_i\odot w)\right)\right)%
\stackrel{(\textbf{IDA3})}{=}%
{\sim}\,{\sim} \left(\bigsqcup_{i\in I} \left(y_i\odot w)\right)\right)\\
&={\sim}\,{\sim} \left(\left(\bigsqcup_{i\in I} y_i\right)\odot w\right).%
\end{align*}
We conclude that $\bigsqcup_{i\in I} x_i\equiv \bigsqcup_{i\in I} y_i$.
\end{proof}

We now introduce the concept of a semi-Foulis dynamic algebra, which extends the notion of an involutive generalized dynamic algebra by requiring the associated set of closed elements (elements of $\widetilde{K}$) to form a complete orthomodular lattice.

\medskip 

\begin{definition}\label{def:igda}
A \emph{semi-Foulis  dynamic algebra} is an involutive generalized dynamic algebra \(\mathfrak{K} = (K, \bigsqcup, \odot, {}^{*}, {\sim}, e)\) satisfying the following condition:
\medskip 
\begin{enumerate}[label=(\textbf{SFDA}),leftmargin=1.57905cm]
    \item $\mathfrak{\widetilde K}=(\widetilde K , {\preceq } , ^{\perp})$ is a complete orthomodular lattice. 
\end{enumerate}
\end{definition}
\medskip
A Foulis quantale is clearly a semi-Foulis  dynamic algebra. 
Moreover, for a  semi-Foulis  dynamic algebra $\mathfrak{K}$, standard Sasaki projections on $\widetilde{\mathfrak{K}}$ are exactly $\mathfrak{K}$-module actions on $\widetilde{\mathfrak{K}}$ induced by elements of $\widetilde K$.

\medskip

\begin{lemma} \label{Sasboth} Let 
    \(\mathfrak{K} = (K, \bigsqcup, \odot, {}^{*}, {\sim}, e)\) 
    be a semi-Foulis  dynamic algebra, $u\in \widetilde K$. 
    Then $u\bullet (-)=\pi_u$.
\end{lemma}
\begin{proof} Since $(\widetilde K , {\preceq } , {}^{\perp})$ is a complete orthomodular lattice, for 
all $v\in \widetilde K$, 
    $$
u\bullet v=%
{\sim}\, {\sim} ({\sim}\, {\sim} u\odot v)\stackrel{(\textbf{IDA5})}{=}%
   {\sim}\left({\sim} u\sqcup {\sim}({\sim}u\sqcup v)\right)%
   =u\wedge (u^{\perp}\vee v) =\pi_u(v).$$
\end{proof}

\medskip

\begin{remark}
   Note that ${\sim}$ is an orthocomplementation on $\mathfrak{\widetilde K}$, but not on $\mathfrak K$.
\end{remark}
\medskip

We now present a construction detailing how to build an involutive unital quantale, $\mathscr{P}(L_\mathcal{M})$, from a complete orthomodular lattice $\mathcal{M}$ and an involutive submonoid $L_{\mathcal M}$ of the Foulis quantale $\mathbf{Lin}(\mathcal M)$, which specifically yields an example of an involutive generalized dynamic algebra.

\medskip
\begin{construction}\label{constrpm}
Let $\mathcal M=(M,\le,{}^\perp)$ be a complete orthomodular lattice, and let   
$L_{\mathcal M}=(L_{M}, \circ, {}^{*}, \mathrm{id}_M)$ be an involutive submonoid of the Foulis quantale 
$\mathbf{Lin}(\mathcal M)=(\mathbf{Lin}(M),\bigsqcup,\circ, {}^{*}, {}^{\perp}, \mathrm{id}_M)$. Assume that  
\[
    \{\pi_m:M\to M \mid m\in M\}\ \subseteq\ L_M\ \subseteq\ 
    \mathbf{Lin}( M),
    \qquad \pi_m(b)=m\wedge(m^\perp\vee b).
  \]

\medskip
\noindent
\emph{Remark.}
At this stage, $L_M$ is \emph{part of the input data}  associated
with $\mathcal M$ and is \emph{not} uniquely determined by $\mathcal M$ in general.
Therefore,  the quantale $\mathscr P(L_{\mathcal M})$, constructed below, depends on the pair $(\mathcal M,L_M)$.

When a canonical choice is needed (see Subsection \ref{sec:concrete-T}, Proposition \ref{prop:T-axioms-hold}), we will take the \emph{least} involutive submonoid of $\mathbf{Lin}(\mathcal M)$ containing all Sasaki projections, namely
\[
L_M^{\mathrm{can}}
:=\bigcap\Big\{\,S\subseteq \mathbf{Lin}(\mathcal M)\ \Big|\ 
S \text{ is an involutive submonoid and } \{\pi_m\mid m\in M\}\subseteq S\,\Big\},
\]
and then we set $L_M:=L_M^{\mathrm{can}}$.

\medskip

We put $\mathscr{P}(L_\mathcal{M})=\bigl(\mathscr{P}(L_{M}) ,\bigcup,\odot, ^ *, {\sim}, \{\mathrm{id}_M\})$, where 
\medskip

 \begin{enumerate}
     \item $\mathscr{P}( L_{M}) $ is the set of all subsets of 
     $L_{M}$,
     
     \item $\bigcup$ is the union of subsets of $L_{M}$, 
     
     \item $\odot$ is a binary operation on $\mathscr P( L_{ M} )$ defined by
  \[
    A\odot B=\{\,a\circ b \mid a\in A,\ b\in B\,\},
  \]
 
  \item ${}^{*}$ is a unary operation on $\mathscr P(L_{M})$ defined by
  \[
    A^{*}=\{\,a^{*}\mid a\in A\,\},
  \]

   \item ${\sim}$ is a unary operation on $\mathscr P(L_{M})$ defined by 
  \[
    {\sim} A=\Big\{\,\pi_{(\,\bigvee_{a\in A} a(1)\,)^{\!\perp}}\,\Big\}.
  \]

 It is well known that $\bigl(\mathscr{P}(L_{M}) ,\bigcup,\odot, ^ *, \{\mathrm{id}_M\})$ is  the free involutive unital quantale over involutive monoid $L_{\mathcal M}$ (see \cite{KP}).
 \end{enumerate}
 \end{construction}

\medskip

\begin{remark}\label{rem:canonical-Si}
The canonical choice \(L_M^{\mathrm{can}}\) introduced above coincides with
\(\mathbf{Si}(M)\) from Example~\ref{Sasaki is Involutive}. Indeed,
\(\mathbf{Si}(M)\) is an involutive submonoid of \(\mathbf{Lin}(\mathcal M)\)
containing all Sasaki projections. Conversely, if
\(S\subseteq \mathbf{Lin}(\mathcal M)\) is any involutive submonoid containing
all Sasaki projections, then \(S\) contains every finite composition of Sasaki
projections, hence \(\mathbf{Si}(M)\subseteq S\). Therefore,
\[
\begin{aligned}
L_M^{\mathrm{can}}
&=\bigcap\Big\{\,S\subseteq \mathbf{Lin}(\mathcal M)\ \Big|\
S \text{ is an involutive submonoid and }
\{\pi_m\mid m\in M\}\subseteq S\,\Big\} \\
&=\mathbf{Si}(M).
\end{aligned}
\]
\end{remark}

The following lemma formalizes the behavior of the operation (${\sim}{\sim}$)  within the constructed quantale $\mathscr{P}(L_\mathcal{M})$. It provides an explicit formula for ${\sim}{\sim} A$ in terms of the join of the initial values of all operators in the set $A$.

\medskip

\begin{lemma}\label{lemma4}
Let ${\mathcal{M}} = (M,\le,{}^{\bot})$ be a complete orthomodular lattice, 
$L_{\mathcal M}$ an involutive submonoid of \/ $\mathbf{Lin}(\mathcal M)$ containing all Sasaki projections, 
and let 
\(A\in \mathscr{P}( L_\mathcal{M})\).
Then
\[
{\sim}{\sim} A \;=\; \Bigl\{\,\pi_{\;\bigvee_{a\in A} a(1)}\Bigr\}.
\]
\end{lemma}

\begin{proof}
By definition of ${\sim}$ on $\mathscr{P}( L_\mathcal M)$,
\[
{\sim} A \;=\; \Bigl\{\pi_{\bigl(\bigvee_{a\in A} a(1)\bigr)^{\bot}}\Bigr\}.
\]
Applying $\sim$ again to the singleton set yields
\[
{\sim}{\sim} A
={\sim}\Bigl\{\pi_{(\bigvee_{a\in A} a(1))^{\bot}}\Bigr\}
=\Bigl\{\pi_{\bigl(\,\bigl(\pi_{(\bigvee_{a\in A} a(1))^{\bot}}\bigr)(1)\,\bigr)^{\bot}}\Bigr\}.
\]
It remains to compute $\bigl(\pi_m\bigr)(1)$ for $m=(\bigvee_{a\in A} a(1))^{\bot}$. 
Since $\pi_m(x)=m\wedge(m^{\bot}\vee x)$, we obtain 
\[
\pi_m(1)=m\wedge(m^{\bot}\vee 1)=m\wedge 1=m.
\]
Therefore
\[
\bigl(\pi_{(\bigvee_{a\in A} a(1))^{\bot}}\bigr)(1)
=(\bigvee_{a\in A} a(1))^{\bot},
\]
and double orthocomplementation yields
\[
{\sim}{\sim} A
=\Bigl\{\pi_{\bigl((\bigvee_{a\in A} a(1))^{\bot}\bigr)^{\bot}}\Bigr\}
=\Bigl\{\pi_{\;\bigvee_{a\in A} a(1)}\Bigr\}.
\]
Completeness of $\mathcal{M}$ guarantees the join $\bigvee_{a\in A} a(1)$ exists, and since all Sasaki projections lie in $ L_M$, the resulting singleton is indeed an element of $\mathscr{P}( L_\mathcal M)$.
\end{proof}

\medskip

The following lemma confirms that the structure $\mathscr{P}(L_{\mathcal{M}})$, derived from a complete orthomodular lattice $\mathcal{M}$ and its associated operations, satisfies all the defining properties of an involutive generalized dynamic algebra.

\medskip

\begin{lemma}\label{PMisIDA}
     Let $\mathcal M=(M,\le,{}^\perp)$ be a complete orthomodular lattice 
     and $L_{\mathcal M}$ an involutive submonoid 
    of \/ $\mathbf{Lin}(\mathcal M)$ containing all Sasaki projections. 
     Then the structure $\mathscr{P}( L_\mathcal{M})=\bigl(\mathscr{P}( L_{M}) ,\bigcup,\odot, ^ *, {\sim}, \{\mathrm{id}_M\})$ is an involutive generalized dynamic algebra.
 \end{lemma}
 \begin{proof} It is enough to check conditions {(\textbf{IDA2})}-{(\textbf{IDA5})}. 

 \medskip

\noindent{(\textbf{IDA2})}: Let $X, Y\in \mathscr{P}( L_\mathcal{M})$. Since  every element $x$ of $L_\mathcal{M}$ is join-preserving we compute:
\begin{align*}  
{\sim}(X\odot {\sim}\, {\sim} Y)&\stackrel{L. \ref{lemma4}}{=}%
{\sim}\left(X\odot \Bigl\{\,\pi_{\;\bigvee_{a\in Y} a(1)}\Bigr\}\right)=%
{\sim}\left(\bigcup\{x\circ \pi_{\;\bigvee_{a\in Y} a(1)}\mid x\in X\}\right)\\
   &=\Bigl\{\,\pi_{\;\left(\bigvee_{x\in X} x\circ \pi_{\;\bigvee_{a\in Y} a(1)}(1)\right)^{\perp}}\,\Bigr\}%
   =\Bigl\{\,\pi_{\;\left(\bigvee_{x\in X} x(\pi_{\;\bigvee_{a\in Y} a(1)}(1))\right)^{\perp}}\,\Bigr\}\\
&=\Bigl\{\,\pi_{\;\left(\bigvee_{x\in X} x(\bigvee_{a\in Y} a(1))\right)^{\perp}}\,\Bigr\}%
=\Bigl\{\,\pi_{\;\left(\bigvee_{x\in X}\bigvee_{a\in Y}  x(a(1))\right)^{\perp}}\,\Bigr\}\\
&=\Bigl\{\,\pi_{\;\left(\bigvee_{z\in X\odot Y}  z(1)\right)^{\perp}}\,\Bigr\}=
    {\sim}(X\odot Y).
\end{align*}

   \medskip

\noindent{(\textbf{IDA3})}: Let $X_i\in \mathscr{P}( L_\mathcal{M})$, $i\in I$. 
 We compute:
 \begin{align*}
    {\sim}\Big(\bigcup_{i \in I}  \,{\sim}\, {\sim} X_{i}\Big)&%
    \stackrel{L. \ref{lemma4}}{=}%
{\sim}\Big(\bigcup_{i \in I}
    \Bigl\{\,\pi_{\;\bigvee_{a\in X_i} a(1)}\Bigr\}\Big)=%
    \Big\{\,\pi_{\left(\,\bigvee\{\bigvee a(1)\mid a\in X_i, i\in I\}\right)^{\!\perp}}\,\Big\}\\
    &=\Big\{\,\pi_{\left(\,\bigvee\{\bigvee a(1)\mid a\in \bigcup_{i\in I}X_i\}\right)^{\!\perp}}\,\Big\}=%
    \,{\sim}\Big(\bigcup_{i \in I}  \, X_{i}\Big)
 \end{align*}
   {(\textbf{IDA4})}: We must check $({\sim} X)^{*} = {\sim} X$ for every  $X\in \mathscr{P}( L_\mathcal{M})$. Using the fact 
   that every Sasaki projection is self-adjoint, we compute:
 \begin{align*}
 ({\sim} X)^{*}=%
 \left(\Bigl\{\pi_{\bigl(\bigvee_{a\in X} a(1)\bigr)^{\bot}}\Bigr\}\right)^{*}=%
 \Bigl\{\left(\pi_{\bigl(\bigvee_{a\in X} a(1)\bigr)^{\bot}}\right)^{*}\Bigr\}=%
 \Bigl\{\pi_{\bigl(\bigvee_{a\in X} a(1)\bigr)^{\bot}}\Bigr\}={\sim} X.
\end{align*}

\medskip

\noindent{(\textbf{IDA5})}: Let $X, Y\in \mathscr{P}( L_\mathcal{M})$. We compute:
\begin{align*}  
{\sim\, {\sim} }(X\odot  Y)&\stackrel{%
(\textbf{IDA2})}{=}%
{\sim\, {\sim} }(X\odot {\sim}\, {\sim} Y)%
\stackrel{L. \ref{lemma4}}{=}%
{\sim\, {\sim} }\left(%
\Bigl\{\,\pi_{\;\bigvee_{a\in X} a(1)}\Bigr\}\odot \Bigl\{\,\pi_{\;\bigvee_{b\in Y} b(1)}\Bigr\}\right)\\
&=%
{\sim\, {\sim} }\left(%
\Bigl\{\,
\pi_{\;\bigvee_{a\in X} a(1)}\Bigr\}\odot %
\pi_{\;\bigvee_{b\in Y} b(1)}\Bigr\}\right)\\
&\stackrel{L. \ref{lemma4}}{=}%
\Bigl\{\,\pi_{\;\left(\bigvee_{a\in X} a(1)\right)\wedge \left((\bigvee_{b\in Y} b(1))\vee (\bigvee_{a\in X} a(1))^{\perp}\right)}\Bigr\}\\
\end{align*}
and 
\begin{align*}
\sim\left(\sim X\sqcup \right.&\left.\sim (\sim X\sqcup Y)\right)\\
&=%
\sim\left(\Big\{\,\pi_{(\,\bigvee_{a\in X} a(1)\,)^{\!\perp}}\,\Big\}%
\sqcup \sim \left( \Big\{\pi_{(\,\bigvee_{a\in X} a(1)\,)^{\!\perp}}%
\Big\} \sqcup  \Big\{\pi_{\;\bigvee_{b\in Y} b(1)}\Big\}
\right)
\right)\\
&=%
\sim\left(\Big\{\,\pi_{(\,\bigvee_{a\in X} a(1)\,)^{\!\perp}}\,\Big\}%
\sqcup \sim \left( \Big\{\pi_{(\,\bigvee_{a\in A} a(1)\,)^{\!\perp}}, %
\pi_{\;\bigvee_{b\in Y} b(1)}\Big\}
\right)
\right)\\
&=%
\sim\left(\Big\{\,\pi_{(\,\bigvee_{a\in X} a(1)\,)^{\!\perp}}\,\Big\}%
\sqcup \Big\{\pi_{\left(\,\bigvee_{a\in X} a(1)\,)^{\!\perp}%
\vee \bigvee_{b\in Y} b(1)\right)^{\perp}}\Big\}\right)\\
&=%
\sim\left(\Big\{\,\pi_{(\,\bigvee_{a\in X} a(1)\,)^{\!\perp}}\, , \pi_{\left(\,\bigvee_{a\in X} a(1)\,)^{\!\perp}%
\vee \bigvee_{b\in Y} b(1)\right)^{\perp}}\Big\}
\right)\\
&{=}%
\Bigl\{\,\pi_{\;\left(\bigvee_{a\in X} a(1)\right)\wedge \left((\bigvee_{b\in Y} b(1))\vee (\bigvee_{a\in X} a(1))^{\perp}\right)}\Bigr\}\\
\end{align*}
    \end{proof}

\medskip 

The following proposition shows that, for any complete
orthomodular lattice $\mathcal{M}$ and any involutive submonoid
$L_{\mathcal{M}}$ of the Foulis quantale $\mathbf{Lin}(\mathcal{M})$ containing all Sasaki projections, the
structure $\widetilde{\mathscr{P}(L_{\mathcal{M}})}$ is isomorphic to
$\mathcal{M}$.

\medskip

\begin{proposition}\label{prop:chi-iso-finitary}
Let ${\mathcal{M}} = \left( {{M},{ \le },\mathop {}\nolimits^{{ \bot }} } \right)$ be a complete orthomodular lattice, 
$L_{\mathcal M}$ an involutive submonoid of \/ $\mathbf{Lin}(\mathcal M)$ containing all Sasaki projections, and $\mathscr{P}( L_\mathcal{M})=\bigl(\mathscr{P}( L_{M}) ,\bigcup,\odot, ^ *, {\sim}, \{\mathrm{id}_M\})$  the corresponding  involutive generalized dynamic algebra.
Then $ \widetilde {\mathscr{P}( L_\mathcal{M})} =%
\left(\{{\sim} W \mid W \in \mathscr{P}( L_\mathcal{M})\},%
{\preceq } , {\sim}\right)$ is a complete  orthomodular lattice closed under ${\sim}$, and
the map $\delta \colon \mathcal{M} \to {\widetilde {\mathscr{P}( L_\mathcal{M})}}$, defined by $\delta(m)=\{\pi_{m}\}$ for each $m \in M$, is an order-isomorphism such that  $\delta(m^{{ \bot }} )={\sim} \delta(m)$.
\end{proposition}
\begin{proof}First, we show that $\delta$ is a bijective map.

\medskip \noindent 
\noindent \textit{Injectivity}:
Let $m,n \in M$ and assume $\delta(m)=\delta(n)$. By the definition of $\delta$, this implies:
\[ \{\pi_m\}=\{\pi_n\}. \]
Thus, $\pi_{m} = \pi_{n}$. Since $\pi_{x}(1) = x$ for any $x \in M$, it follows that
\[ m = \pi_{m}(1) = \pi_{n}(1) = n. \]

Therefore, $\delta$ is injective.

\medskip\noindent
\noindent \textit{Surjectivity:}
Let $Y\in \widetilde{\mathscr P(L_{\mathcal M})}$. Then $Y={\sim} W$ for some $W\subseteq L_{\mathcal M}$.
By the explicit form of ${\sim}$ on arbitrary subsets,
\[
{\sim} W \;=\; \Bigl\{\ \pi_{\ \left(\bigvee_{a\in W} a(1)\right)^{\perp}}\ \Bigr\}.
\]

Set $m:=(\bigvee_{a\in W} a(1))^{\perp}\in M$. Then $Y=\{\pi_m\}=\delta(m)$.
Hence $\delta$ is onto.

\medskip\noindent{}From Theorem \ref{thm:K-module} we know that $\widetilde{\mathscr P(L_{\mathcal M})}$ 
is a complete lattice such that ${\sim}\,{\sim}\,{\sim} X={\sim} X$ for every $X\in \widetilde{\mathscr P(L_{\mathcal M})}$.

\medskip\noindent 
\noindent \textit{Preservation of finite joins}: 
Let $m,n \in M$. We need to show that $\delta(m \vee n) = \delta(m) \vee \delta(n)$.
In $\widetilde {\mathscr{P}( L_\mathcal{M})}$, the join of $\delta(m)=\{\pi_m\}$ and $\delta(n)=\{\pi_n\}$ is defined as:
\[ \delta(m)\vee\delta(n) = {\sim}{\sim}\left(\{\pi_m\}\cup\{\pi_n\}\right) = {\sim}{\sim}\{\pi_m,\pi_n\}. \]
By Lemma~\ref{lemma4}, we have:
\[ {\sim}{\sim}\{\,\pi_m,\pi_n\} = \{\,\pi_{\left(\,\pi_m(1)\vee \pi_n(1)\right)}\}. \]
Since $\pi_x(1) = x$, this further simplifies to:
\[ \{\,\pi_{\left(\,\pi_m(1)\vee \pi_n(1)\right)}\} = \{\,\pi_{m\vee n}\}. \]
Finally, by the definition of $\delta$:
\[ \{\,\pi_{m\vee n}\} = \delta(m\vee n). \]
Therefore, $\delta(m \vee n) = \delta(m) \vee \delta(n)$, demonstrating that $\delta$ preserves the join operation.

\medskip\noindent 
\noindent \textit{Preservation of Orthocomplement}:
Let $m \in M$. We need to show that $\delta(m^\perp) = {\sim}\delta(m)$.
By the definition of $\delta$:
\[ \delta(m^\perp) = \{\,\pi_{m^\perp}\,\}. \]
Now, consider ${\sim}\delta(m)$:
\[ {\sim}\delta(m) = {\sim}\{\pi_m\}. \]
Using the definition of the ${\sim}$ operation on a singleton set and given that $\pi_m(1) = m$, we have:
\[ {\sim}\{\pi_m\} = \left\{\,\pi_{\left(\pi_m(1)\right)^{\perp}}\right\} = \{\,\pi_{m^{\perp}}\,\}. \]
Combining these steps, we obtain:
\[ \delta(m^\perp) = \{\,\pi_{m^\perp}\,\} = {\sim}\{\,\pi_{m}\,\} = {\sim}\delta(m). \]
Thus, $\delta$ preserves the orthocomplement operation and since it is an order-preserving bijection we conclude that %
$\left(\widetilde {\mathscr{P}( L_\mathcal{M})},\preceq, {\sim}\right)$ is 
a complete orthomodular lattice isomorphic to ${\mathcal{M}}$. 
\end{proof}

The final result of the previous steps culminates in the following theorem, which proves that the constructed involutive generalized dynamic algebra $\mathscr{P}(L_\mathcal{M})$ is, in fact, a semi-Foulis dynamic algebra, thereby establishing a concrete example of this algebraic structure.

\medskip

\begin{theorem}\label{theorem:chi-iso-semifoulis}
 Let $\mathcal M=(M,\le,{}^\perp)$ be a complete orthomodular lattice 
     and $L_{\mathcal M}$ an involutive submonoid 
    of \/ $\mathbf{Lin}(\mathcal M)$ containing all Sasaki projections. 
     Then the involutive generalized dynamic algebra $\mathscr{P}( L_\mathcal{M})=\bigl(\mathscr{P}( L_{M}) ,\bigcup,\odot, ^ *, {\sim}, \{\mathrm{id}_M\})$   is a semi-Foulis dynamic algebra.
\end{theorem}

\section{The category $\mathbb{\mathscr{T}ODA}$ of $\mathscr{T}$-based orthomodular dynamic algebras}\label{sectionTODA}

The first part of this section introduces $\mathscr{T}$-based orthomodular dynamic algebras, which expand upon the existing structures of orthomodular dynamic algebras and involutive  generalized dynamic algebras. (See \cite{KRSZ} for details on orthomodular dynamic algebras.)

In the second part of this section we present a concrete instance of the abstract functor $\mathscr{T}:\mathbb{IDA}\longrightarrow\mathbb{IM}$.

\subsection{$\mathscr{T}$-based orthomodular dynamic algebras}

\begin{definition}\label{def:IM}
Let $\mathbb{IM}$ be the category of \emph{involutive monoids}, defined as follows.

\begin{enumerate}
\item \textbf{Objects.}
An object of $\mathbb{IM}$ is a quadruple
\[
\mathbf{M}=(M,\cdot,e,{}^{*})
\]
such that $(M,\cdot,e)$ is a monoid and ${}^{*}:M\to M$ is a map satisfying, for all
$x,y\in M$,
\[
(x^{*})^{*}=x,
\qquad
(x\cdot y)^{*}=y^{*}\cdot x^{*}.
\]

\item \textbf{Morphisms.}
Given objects $\mathbf{M}=(M,\cdot,e,{}^{*})$ and $\mathbf{N}=(N,\odot,u,{}^{\dagger})$,
a morphism $f:\mathbf{M}\to\mathbf{N}$ in $\mathbb{IM}$ is a map
\[
f:M\to N
\]
such that, for all $x,y\in M$,
\[
f(x\cdot y)=f(x)\odot f(y),
\qquad
f(e)=u,
\qquad
f(x^{*})=f(x)^{\dagger}.
\]
Composition and identities are the same as those of $\mathbf{Set}$. Obviously, by definitions, for any $\mathbf{M}=(M,\cdot,e,{}^{*}) \in \mathbb{IM}$, it holds that $e^{*}=e$.

\end{enumerate}
\end{definition}

    \medskip 
    \noindent
    Consider the categories $\mathbb{IM}$, $\mathbb{IDA}$ and let   
    $\mathscr{T}$ 
    be a functor from $\mathbb{IDA}$ to 
    $\mathbb{IM}$ such that: 

    \medskip 
    \begin{enumerate}[label=(\textbf{T\arabic*}), leftmargin=1cm]
        \item $\widetilde{K} \subseteq \mathscr{T}({K}) \subseteq K%
    \text{ and } \mathscr{T}(\mathfrak{K})=%
    (\mathscr{T}({K}),\odot, {}^{*}, e) 
    \text{ is an involutive submonoid of } \mathfrak{K}$
         for every involutive  generalized dynamic algebra $\mathfrak{K} \in \mathbb{IDA}$. 
         \item For every semi-Foulis  dynamic algebra $\mathfrak{K} \in \mathbb{IDA}$ such that, for any $s,t\in \mathscr{T}( K)$, $s = t$ if and only if $s \equiv t$, 
         the involutive monoids $\mathscr{T}(\mathfrak{K})$ and $\mathscr{T}(\mathbf{Lin}(\widetilde{\mathfrak{K}}))$  are isomorphic in $\mathbb{IM}$ via a canonical map $\nu_{\mathfrak{K}}$ 
         defined by $\nu_{\mathfrak{K}}(k)=k \bullet (-)$ for all $k\in \mathscr{T}(\mathfrak{K})$.
         \item For every complete orthomodular lattice  $\mathcal M$, the involutive monoids  
         $\mathscr{T}(\mathbf{Lin}(\mathcal M))$ and 
         $\mathscr{T}\left(\mathscr{P}\left(\mathscr{T}(\mathbf{Lin}(\mathcal M))\right)\right)$ are isomorphic in $\mathbb{IM}$ via a canonical map $\mu_{\mathcal M}$ defined by $\mu_{\mathcal M}(f)=\{f\}$ for all $f\in \mathscr{T}(\mathbf{Lin}(M))$.
         \item For every morphism $f\colon \mathfrak{K}_1\to \mathfrak{K}_2$ of involutive  generalized dynamic algebras, $\mathscr{T}(f)$ is a restriction of $f$ to $\mathscr{T}(\mathfrak{K}_1)$.
    \end{enumerate}
    In particular  $\mathscr{T}$ preserves isomorphisms:
        \[
           \mathfrak{K} \cong \mathfrak{L} \text{ in } \mathbb{IDA} \implies \mathscr{T}(\mathfrak{K}) \cong %
           \mathscr{T}(\mathfrak{L}) \text{ in } \mathbb{IM}.
        \]

We assume these properties for the functor $\mathscr{T}$ throughout the rest of this paper.

\medskip

  \begin{definition}\label{finitary-goda}
A \textit{$\mathscr{T}$-based orthomodular dynamic algebra} is an involutive generalized dynamic algebra $\mathfrak{K} = (K , {\bigsqcup } , {\odot } , {}{^*}, {\sim}, e)$ satisfying the following conditions:
  
\begin{enumerate}[label=(\textbf{TODA\arabic*}), leftmargin=2cm]
    
    \item $\mathfrak{\widetilde K}=(\widetilde K , {\preceq } , {}^{\perp} )$ is a complete orthomodular lattice. 
    \item If a set $A$ meets these criteria:
    \begin{enumerate}[label=(\alph*), leftmargin=0.69cm]
        \item $\mathscr{T} \left( K \right) \subseteq A \subseteq K$.
        \item $A$ is closed under both the $\odot$ and ${{}^*}$ operations.
        \item $A$ is closed under $\bigsqcup$, meaning that for any subset $B$ of $A$,
        its join ($\bigsqcup B$) must also be in $A$.
    \end{enumerate}
    Then, $A$ must be equal to $K$. This condition ensures minimality.

    \item For any subsets $S , T \subseteq \mathscr{T} (K)$, their joins are equal ($\bigsqcup S = \bigsqcup T$) if and only if the sets themselves are equal ($S = T$). This ensures proper set equality.

    \item For any $s,t\in \mathscr{T}( K)$, $s = t$ if and only if $s \equiv t$. This condition guarantees completeness.
%
\end{enumerate}

\end{definition}
Clearly, every $\mathscr{T}$-based orthomodular dynamic algebra is a semi-Foulis dynamic algebra.
\medskip

We will now proceed to prove two instrumental lemmas. The first establishes the existence of a normal form for every element within a $\mathscr{T}$-based orthomodular dynamic algebra.
\medskip

\begin{lemma}\label{lemma2}
Let $\mathfrak{K}= (K, {\bigsqcup } , {\odot }, { {}^{{*}}} , {\sim},e)$ be a $\mathscr{T}$-based
orthomodular dynamic algebra. For each $k\in K$, there exists a unique
indexed set $\{\,s_i \mid i\in I\,\}\subseteq \mathscr{T} (K)$
such that
\[
k \;=\; \bigsqcup \{\,s_i \mid i\in I\,\}
\qquad\text{(equivalently, }k=\bigsqcup_{i\in I} s_i\text{).}
\]
\end{lemma}
\begin{proof}
Let
\[
A \;=\; \Bigl\{\, k\in K \;\Bigm|\; \exists I\ \exists \{t_i \mid i\in I\}\subseteq\mathscr{T} (K)
\text{ with } k \;=\; \bigsqcup \{t_i \mid i\in I\}\,\Bigr\}.
\]
Take arbitrary $k,l\in A$ with witnesses
\[
k=\bigsqcup \{t_i \mid i\in I\}, \qquad
l=\bigsqcup \{s_j \mid j\in J\},
\]
where $\{t_i \mid i\in I\},\{s_j \mid j\in J\}\subseteq \mathscr{T} (K)$.

\begin{enumerate}
\renewcommand{\labelenumi}{\alph{enumi}.}
\item (\emph{Base}) For any $v\in \mathscr{T} (K)$ we have
\[
v \;=\; \bigsqcup \{x \mid x\in \{v\}\},
\]
and since $\{v\}\subseteq \mathscr{T} (K)$, it follows $v\in A$.
Hence $\widetilde K \subseteq A \subseteq K$.

\item (\emph{Closure under operations})
\begin{itemize}
\item[$\odot$\ :]
\[
k\odot l
= \Bigl(\bigsqcup \{t_i \mid i\in I\}\Bigr)\odot
  \Bigl(\bigsqcup \{s_j \mid j\in J\}\Bigr)
= \bigsqcup \{\, t_i\odot s_j \mid i\in I,\ j\in J \,\},
\]
by complete distributivity of $\odot$ in each coordinate. Since
$\{\, t_i\odot s_j \mid i\in I,\ j\in J \,\}\subseteq \mathscr{T} (K)$,
we get $k\odot l\in A$.

\item[${}^{*}$\ :]
\[
k^{*}
= \Bigl(\bigsqcup \{t_i \mid i\in I\}\Bigr)^{*}
= \bigsqcup \{\, t_i^{*} \mid i\in I \,\},
\]
by complete distributivity of $^{*}$. As
$\{\, t_i^{*} \mid i\in I \,\}\subseteq \mathscr{T} (K)$, we have $k^{*}\in A$.
\end{itemize}

\item (\emph{Closure under arbitrary joins})
For a family $\{k_\alpha\}_{\alpha\in \Lambda}$ with
$k_\alpha=\bigsqcup\{t^{(\alpha)}_i \mid i\in I_\alpha\}$, $\{t^{(\alpha)}_i \mid i\in I_\alpha\}\subseteq \mathscr{T} (K)$, one has
\[
\bigsqcup_{\alpha\in \Lambda} k_\alpha
= \bigsqcup \Bigl\{\, t^{(\alpha)}_i \ \Bigm|\ \alpha\in\Lambda,\ i\in I_\alpha \Bigr\}\in A.
\]
\end{enumerate}
Thus $A$ satisfies the hypotheses of (\textbf{TODA2}), hence $A=K$.

For uniqueness, suppose
\[
k \;=\; \bigsqcup \{t_i \mid i\in I\} \;=\; \bigsqcup \{u_r \mid r\in R\},
\qquad
\{t_i \mid i\in I\},\{u_r \mid r\in R\}\subseteq \mathscr{T} (K).
\]
By (\textbf{TODA3}) (uniqueness of the join-decomposition from
$\mathscr{T} (K)$), we conclude
\[
\{t_i \mid i\in I\} \;=\; \{u_r \mid r\in R\}.
\]
Hence the representation is unique.
\end{proof}

\medskip

The following construction defines a related algebraic structure, $\mathscr{P}(\mathscr{T} (\mathfrak{K}))$, by considering the involutive submonoid of the test set $\mathscr{T}(K)$ of a $\mathscr{T}$-based orthomodular dynamic algebra $\mathfrak{K}$, and explicitly using the properties of the bijective map $h_{\mathfrak{K}}$ to induce new operations ($\sqdot$, ${}^{\star}$, $\sqtilde$).

\medskip

\begin{construction}\label{constrisom}
Let $\mathfrak{K}= (K, {\bigsqcup } , {\odot }, { {}^{{*}}} , {\sim},e)$ be a $\mathscr{T}$-based orthomodular dynamic algebra.
\medskip

Denote by $h_{\mathfrak{K}}$ a map from $K$ into 
    $\mathscr{P}(\mathscr{T} (K))$ defined by 
    $h_{\mathfrak{K}}(v)=\{w\in \mathscr{T} (K)\mid w\leq v\}$ for every $v\in K$. Clearly, $h_{\mathfrak{K}}$  correctly maps into its range and 
    from Lemma \ref{lemma2} we know that $\bigsqcup h_{\mathfrak{K}}(v)=v$ 
    and $ h_{\mathfrak{K}}(\bigsqcup A)=A$ for every $v\in K$ 
    and every $A\in \mathscr{P}(\mathscr{T} (K))$.
    We immediately obtain $h_{\mathfrak{K}}$ is both injective and surjective. 
\medskip

We put $\mathscr{P}(\mathscr{T} (\mathfrak{K}))=%
\bigl(\mathscr{P}(\mathscr{T} (K)),\bigcup,\sqdot, {}^{\star}, \sqtilde, \{e\})$, where 
\medskip

 \begin{enumerate}
     \item $\mathscr{P}(\mathscr{T} (K))$ is the set of all subsets of 
     $\mathscr{T} (K)$,
     
     \item $\bigcup$ is the union of subsets of $\mathscr{T} (K)$, 
     
   \item $\sqdot$ (the \emph{lifted product}) is a binary operation on $\mathscr{P}(\mathscr{T}(K))$ defined by
\[
  A\,\sqdot\,B
  =h_{\mathfrak{K}}\!\left(\bigsqcup A \odot \bigsqcup B\right).
\]

\item ${}^{\star}$ (the \emph{lifted involution}) is a unary operation on $\mathscr{P}(\mathscr{T}(K))$ defined by
\[
  A^{\star}
  =h_{\mathfrak{K}}\!\left(\left(\bigsqcup A\right)^{*}\right).
\]

\item $\sqtilde$ (the \emph{lifted negation}) is a unary operation on $\mathscr{P}(\mathscr{T}(K))$ defined by
\[
  \sqtilde A
  =h_{\mathfrak{K}}\!\left({\sim}\left(\bigsqcup A\right)\right).
\]

\end{enumerate}
\end{construction}

\medskip 

The following lemma provides a significant characterization of $\mathscr{T}$-based orthomodular dynamic algebras $\mathfrak{K}$, proving it is an atomic lattice whose atoms are precisely the test set $\mathscr{T}(K)$. Crucially, it then establishes an isomorphism between $\mathfrak{K}$ and the constructed structure $\mathscr{P}(\mathscr{T}(\mathfrak{K}))$, confirming that $\mathscr{P}(\mathscr{T}(\mathfrak{K}))$ is a $\mathscr{T}$-based orthomodular  dynamic algebra and that $h_{\mathfrak{K}}$ is the connecting algebraic homomorphism.

\medskip

\begin{lemma}\label{lemma-atomic}
Let $\mathfrak{K}= (K, {\bigsqcup } , {\odot }, { {}^{{*}}} , {\sim},e)$ be a $\mathscr{T}$-based
orthomodular dynamic algebra. Then 
\begin{enumerate}[label={\em(\textrm{{\roman*}})}, leftmargin=1cm]
    \item $\mathfrak{K}$ is an atomic lattice and its atoms are exactly 
    elements of $\mathscr{T} (K)$.
    \item The structure $\mathscr{P}(\mathscr{T} (\mathfrak{K}))=%
\bigl(\mathscr{P}(\mathscr{T} (K)),\bigcup,\sqdot, {}^{\star}, \sqtilde, \{e\})$ is  a $\mathscr{T}$-based
orthomodular dynamic algebra and 
    $\mathfrak{K}$ is isomorphic to $\mathscr{P}(\mathscr{T} (\mathfrak{K}))$ via $h_{\mathfrak{K}}$.
\end{enumerate}
\end{lemma}
\begin{proof} (i): Let $v\in K$ be an atom of $\mathfrak{K}$. Then $v\not=0$. By Lemma \ref{lemma2},
there is a unique non-empty subset $\{t_i \mid i\in I\}\subseteq \mathscr{T} (K)$ such that $v=\bigsqcup_{i\in I}\{t_i\}$. Hence there is $i_0\in I$ such that $t_{i_0}\leq v$. Since $v$ is an atom, $t_{i_0}= v$. 
Therefore, every atom of $K$ lies in $\mathscr{T} (K)$. 

Conversely, assume that $w\in \mathscr{T} (K)$. 
Then $w=\bigvee\{w\}$. 
Clearly, by (\textbf{TODA3}), $w\not=0$ since $0=\bigvee\emptyset$. Assume that there is a
    $v\in K$, $v\leq w$ and $v\not=0$. Then there is $t\in \mathscr{T} (K)$, $t\not=0$  such that $t\leq v$. But $w=w\sqcup t$. Therefore, 
    by the uniqueness of the join-decomposition as shown in Lemma \ref{lemma2}, $w=t\leq v\leq w$. We conclude $v=w$. 
    
    Hence, the atoms of $\mathfrak{K}$ are exactly 
    elements of $\mathscr{T} (K)$ and again by Lemma \ref{lemma2},
    $\mathfrak{K}$ is an atomic lattice.

    \medskip
    \noindent(ii):  Clearly, $h_{\mathfrak{K}}$ is an order-preserving bijection with an inverse $\bigsqcup\colon \mathscr{P}(\mathscr{T} (K)) \to K$, and $(\widetilde K , {\preceq } , {}^{\perp} )$ is a complete orthomodular lattice. Hence the operations $\sqdot, {}^{\star}$ and $ \sqtilde$ are 
    correctly defined and $\mathscr{P}(\mathscr{T} (\mathfrak{K}))$ is a semi-Foulis dynamic algebra. Since conditions 
(\textbf{TODA2})-(\textbf{TODA4}) are preserved by  isomorphisms  of 
involutive  generalized dynamic algebras
    we conclude that $\mathscr{P}(\mathscr{T} (\mathfrak{K}))$ is 
    a $\mathscr{T}$-based orthomodular dynamic algebra.
\end{proof}

To establish the category of $\mathscr{T}$-based orthomodular dynamic algebras, it is necessary to define a morphism between  two such algebras. This will subsequently be referred to as a $\mathscr{T}\mathbb{ODA}$-morphism in the forthcoming definition.

\medskip

\begin{definition}\label{q-morphism}
Let $\mathfrak{K}_1 = (K_1 , {\bigsqcup }_1 , {\odot }_1, { {}^{{* _1}}} , {\sim_1},e_1)$ and $\mathfrak{K}_2 = (K_2 , {\bigsqcup }_2 , {\odot }_2 , { {}^{{* _2}}}, {\sim_2}, e_2)$ be $\mathscr{T}$-based orthomodular dynamic algebras. A $\mathscr{T}\mathbb{ODA}$–morphism $\phi :{\mathfrak{K}_1} \to {\mathfrak{K}_2}$ is a bijective morphism of involutive generalized dynamic algebras  $\phi :{K_1} \to {K_2}$.

\medskip

The category of $\mathscr{T}$-based orthomodular dynamic algebras and $\mathscr{T}\mathbb{ODA}$-morphisms is denoted by $\mathscr{T}\mathbb{ODA}$. Note that $\mathscr{T}\mathbb{ODA}$ is a subcategory of $\mathbb{IDA}$.
\end{definition}
\medskip 
The following lemma addresses the relationship between morphisms of $\mathscr{T}$-based orthomodular dynamic algebras and the corresponding ortholattice isomorphisms on their sets of closed elements. Specifically, it demonstrates that a $\mathscr{T}\mathbb{ODA}$–morphism between two such algebras induces an ortholattice isomorphism between their associated orthomodular lattices of closed elements.
\medskip
\begin{lemma}\label{morisorthoiso}
Let $\mathfrak{K}_1 = (K_1 , {\bigsqcup }_1 , {\odot }_1, { {}^{{* _1}}} , {\sim_1},e_1)$ and $\mathfrak{K}_2 = (K_2 , {\bigsqcup }_2 , {\odot }_2 , { {}^{{* _2}}}, {\sim_2}, e_2)$ be $\mathscr{T}$-based orthomodular dynamic algebras, and let $\phi :{\mathfrak{K}_1} \to {\mathfrak{K}_2}$ be a 
$\mathscr{T}\mathbb{ODA}$–morphism. Then the restriction 
$\widetilde{\phi}$ of $\phi$ to $\widetilde{K_1}$ is an ortholattice 
isomorphism between orthomodular lattices $(\widetilde K_1 , {\preceq }_1 , {}^{\perp_1} )$ and $(\widetilde K_2 , {\preceq }_2 , {}^{\perp_2}  )$.
\end{lemma}
\begin{proof}Let $x\in \widetilde K_1$. Then $x={\sim_1} z$ for some $z\in K_1$. 
We compute:
$$
\phi(x)=\phi({\sim_1} z)={\sim_2} \phi(z)\in K_2.
$$
We conclude that $\widetilde{\phi}\colon \widetilde K_1\to \widetilde K_2$ 
is correctly defined. Since $\phi$ is injective, also its restriction 
$\widetilde{\phi}$ is injective. Let us show that $\widetilde{\phi}$ is surjective. 

Let $u\in \widetilde K_2$. Then $u={\sim_2} v$ for some $v\in K_2$. Since 
$\phi$ is surjective there is $x\in K_1$ such that $\phi(x)=v$. We compute:
$$
\phi({\sim_1} x)={\sim_2} \phi(x)={\sim_2} v=u.
$$
Hence $\widetilde{\phi}$ is surjective. 

\medskip

It remains to show that $\widetilde{\phi}$ is order-preserving. 
Let $x, y\in \widetilde{K_1} $ and $x{\preceq }_1 y$. Then 
$\bigvee_1 \left\{ {x,y} \right\} = y$. We compute:
\begin{align*}
    \bigvee_2 \left\{ {\widetilde{\phi}(x),\widetilde{\phi}(y)} \right\} &=%
   {\sim_2}\,{\sim_2} \bigsqcup_2 \left\{ {\widetilde{\phi}(x),\widetilde{\phi}(y)} \right\} =%
   {\sim_2}\,{\sim_2} \widetilde{\phi}\left(\bigsqcup_1 \left\{ {x,y} \right\}\right)\\%
   &=%
\widetilde{\phi}\left({\sim_1}\,{\sim_1}\bigsqcup_1 \left\{ {x,y} \right\}\right)=%
\widetilde{\phi}\left(\bigvee_1 \left\{ {x,y} \right\}\right)= %
\widetilde{\phi}(y).
\end{align*}
We conclude that $\widetilde{\phi}(x)\,{\preceq }_2\, \widetilde{\phi}(y)$. 
Hence $\widetilde{\phi}$ is order-preserving. 

\medskip

Since $\phi$ preserves ${\sim}$ we conclude that 
$\widetilde{\phi}$ is ${}^{\perp}$-preserving.
\end{proof}

\subsection{A concrete choice of the functor $\mathscr{T}$}
\label{sec:concrete-T}

 In our setting, $\mathscr{T}(\mathfrak{K})$
will be the involutive submonoid of the underlying involutive monoid $\mathfrak{K}$ generated by the distinguished subset
$\widetilde{{K}}$.

\medskip

\begin{definition}[The functor $\mathscr{T}$]\label{def:functor-T}
Let $\mathfrak{K}=(K,\bigsqcup,\odot,{}^*,{\sim},e)$ be an involutive generalized dynamic algebra
and let $\widetilde{{K}}\subseteq K$ be the test set of
$\mathfrak{K}$.

\begin{itemize}
  \item On objects, we define $\mathscr{T}(\mathfrak{K})$ as the involutive
    submonoid of the involutive monoid $(K,\odot,{}^*,e)$ generated by $\widetilde{{K}}$,
    i.e.
    \[
      \mathscr{T}({K})
      \ :=\
      \bigcap\{\,M\subseteq K \mid
        \mathcal M\text{ is an involutive submonoid of }(K,\odot,{}^*,e)
        \text{ and }\widetilde{{K}}\subseteq M\,\}.
    \]
    Equivalently, $\mathscr{T}({K})$ is the smallest subset of $K$
    containing $\widetilde{{K}}$ and closed under unit ($e$), multiplication ($\odot$), and
    involution (${}^*$).

  \item On morphisms, let $f:\mathfrak{K}_1\to\mathfrak{K}_2$ be a morphism
    in $\mathbb{IDA}$, with underlying map $f:K_1\to K_2$. Since
    $f(\widetilde{{K}}_1)\subseteq \widetilde{{K}}_2$ and
    $f$ preserves $\odot$, ${}^*$ and $e$, we have
    $f(\mathscr{T}({K}_1))\subseteq \mathscr{T}({K}_2)$
    (see Proposition~\ref{prop:T-axioms-hold} below), and we define
    \[
      \mathscr{T}(f)\ :=\
      f|_{\mathscr{T}({K}_1)}:
      \mathscr{T}({K}_1)\longrightarrow\mathscr{T}({K}_2).
    \]
\end{itemize}
\end{definition}
Evidently, $\mathscr{T}:\mathbb{IDA}\to\mathbb{IM}$ is indeed a functor.

\medskip
This proposition outlines the fundamental properties of the concrete functor $\mathscr{T}$, which maps an involutive 
generalized dynamic algebra $\mathfrak K$ to its submonoid of testable elements $\mathscr{T}(\mathfrak K)$. It establishes structural relationships, an isomorphism with the test set part of the Foulis quantale, and how the functor acts on morphisms.
\medskip

\begin{proposition}[Basic properties of the concrete functor $\mathscr{T}$]
\label{prop:T-axioms-hold}
Let $\mathscr{T}:\mathbb{IDA}\to\mathbb{IM}$ be given by
Definition~\textup{\ref{def:functor-T}}. Then:

\begin{enumerate}
  \item\label{item:T-basic-1}
  For every $\mathfrak{K}\in\mathbb{IDA}$ we have
  \[
    \widetilde{{K}}\ \subseteq\ \mathscr{T}({K})\ \subseteq\ K,
  \]
  and $\mathscr{T}(\mathfrak{K})$ is an involutive submonoid of
  $(K,\odot,{}^*,e)$.

  \item\label{item:T-basic-2}
  Let $\mathfrak K\in\mathbb{IDA}$ be a semi-Foulis dynamic algebra such that,  for any $s,t\in\mathscr{T}(K)$,
$s=t$ iff $s\equiv t$.

  
  Then the
  canonical left action of $\mathscr{T}({K})$ on the set of tests 
  $\widetilde{{K}}$ induces an 
  isomorphism of
  involutive monoids
  \[
    \nu_{\mathfrak{K}}:\mathscr{T}(\mathfrak{K})\to 
       \mathscr{T}(\mathbf{Lin}(\widetilde{\mathfrak{K}}))=
    L_{\widetilde{\mathfrak{K}}}^{\mathrm{can}},
    \qquad
    \nu_{\mathfrak{K}}(k)\ :=\ k\bullet(-),
  \]
  whose restriction
  \[
    \nu_{\mathfrak{K}}|_{\widetilde{\mathfrak{K}}}:
    \widetilde{\mathfrak{K}}\longrightarrow
    \widetilde{\mathbf{Lin}(\widetilde{\mathfrak{K}})}
  \]
  is an order-preserving bijection between the distinguished orthomodular lattices.

  \item\label{item:T-basic-3}
  For every complete orthomodular lattice $\mathcal{M}$, the involutive
  monoids $\mathscr{T}(\mathbf{Lin}(\mathcal{M}))$ and
  $\mathscr{T}\bigl(\operatorname{\mathscr{P}}(\mathscr{T}(\mathbf{Lin}
  (\mathcal{M})))\bigr)$ are isomorphic in $\mathbb{IM}$ via the
  canonical map
  \[
    \mu_{\mathcal{M}}:\mathscr{T}(\mathbf{Lin}(\mathcal{M}))
       \xrightarrow{\ \cong\ }
       \mathscr{T}\bigl(\mathscr{P}(\mathscr{T}(\mathbf{Lin}
       (\mathcal{M})))\bigr),
    \qquad
    \mu_{\mathcal{M}}(f)\ :=\ \{f\}.
  \]

  \item\label{item:T-basic-4}
  For every morphism $f:\mathfrak{K}_1\to\mathfrak{K}_2$ in
  $\mathbb{IDA}$ we have
  \[
    \mathscr{T}(f)
      \ =\
      f|_{\mathscr{T}(\mathfrak{K}_1)}:
      \mathscr{T}(\mathfrak{K}_1)\longrightarrow\mathscr{T}(\mathfrak{K}_2),
  \]
  i.e.\ $\mathscr{T}(f)$ is the restriction of $f$ to
  $\mathscr{T}(\mathfrak{K}_1)$.
\end{enumerate}
\end{proposition}
\begin{proof}
We check the items one by one.

\smallskip
\noindent\textup{({\ref{item:T-basic-1}}):}
Fix $\mathfrak{K}=(K,\bigsqcup,\odot,{}^*,{\sim},e)\in\mathbb{IDA}$ and put
\[
\mathcal{F}_{\mathfrak{K}}
\ :=\
\Bigl\{\,M\subseteq K \ \Bigm|\ 
M\text{ is an involutive submonoid of }(K,\odot,{}^*,e)
\text{ and }\widetilde{\mathfrak K}\subseteq M\,\Bigr\}.
\]
By Definition~\ref{def:functor-T} we have
\[
\mathscr{T}(\mathfrak{K})
\ =\
\bigcap_{M\in\mathcal{F}_{\mathfrak{K}}} M.
\]
Evidently, $\mathscr{T}(\mathfrak{K})\subseteq K$.
The family $\mathcal{F}_{\mathfrak{K}}$ is nonempty since $K\in\mathcal{F}_{\mathfrak{K}}$,
hence the intersection is well defined.
Intersections of involutive submonoids are involutive submonoids, so
$\mathscr{T}(\mathfrak{K})$ is an involutive submonoid of $(K,\odot,{}^*,e)$
containing $\widetilde{\mathfrak K}$, which proves \textup{({\ref{item:T-basic-1}})}.

\medskip

\noindent\textup{({\ref{item:T-basic-2}}):}
Assume that $\mathfrak K$ is a semi-Foulis dynamic algebra satisfying  the separating condition
\[
(\forall s,t\in \mathscr{T}(\mathfrak K))\qquad
s=t \iff s\equiv t,
\]
where $s\equiv t$ means $(\forall w\in\widetilde{\mathfrak K})\ (s\bullet w=t\bullet w)$.

\smallskip
\noindent\textbf{Step 1: the action homomorphism.}
Recall the canonical left action on tests
\[
\bullet:{\mathfrak K}\times\widetilde{\mathfrak K}\to\widetilde{\mathfrak K},
\qquad
k\bullet v:={\sim}{\sim}(k\odot v),
\]
and define
\[
\Phi:K\to \operatorname{Set}(\widetilde{\mathfrak K},\widetilde{\mathfrak K}),
\qquad
\Phi(k):=k\bullet(-).
\]
Here $\operatorname{Set}(\widetilde{\mathfrak K},\widetilde{\mathfrak K})$ is the set of all endomaps on the set ${\widetilde{K}}$. 
By the left action laws (Theorem~\ref{thm:K-module}), for all $k,\ell\in K$ and $v\in\widetilde{K}$,
\[
\Phi(k\odot\ell)(v)
=(k\odot\ell)\bullet v
= k\bullet(\ell\bullet v)
=\Phi(k)\bigl(\Phi(\ell)(v)\bigr),
\]
and $\Phi(e)=\operatorname{id}_{\widetilde{\mathfrak K}}$.
Hence
\[
\Phi(k\odot\ell)=\Phi(k)\circ\Phi(\ell),
\qquad
\Phi(e)=\operatorname{id}_{\widetilde{\mathfrak K}},
\]
so $\Phi$ is a monoid homomorphism.

\smallskip
\noindent\textbf{Step 2: tests in $\mathbf{Lin}(\widetilde{\mathfrak K})$ and the image of $\mathscr{T}(\mathfrak K)$.}
For each $u\in\widetilde{K}$ define the (ortholattice) Sasaki projection
\[
\pi_u:\widetilde{\mathfrak K}\to\widetilde{\mathfrak K},
\qquad
\pi_u(v):=u\wedge(u^\perp\vee v).
\]
Since $\widetilde{\mathfrak K}$ is a $\mathfrak K$-module by Theorem \ref{thm:K-module},  we  have, for every $u\in\widetilde{K}$, that $\Phi(u)=u\bullet(-)
$ preserves arbitrary joins in $\widetilde{\mathfrak K}$. 
Moreover, from Lemma \ref{Sasboth}, we have that 
$\Phi(u)=u\bullet(-)=\pi_u$. 

Therefore
\[
\Phi(\widetilde{\mathfrak K})
=\{\pi_u\mid u\in\widetilde{\mathfrak K}\}
=\widetilde{\mathbf{Lin}(\widetilde{\mathfrak K})},
\]
where the last equality holds by definition of the distinguished test set
$\widetilde{\mathbf{Lin}(M)}:=\{\pi_m\mid m\in M\}$ and Remark \ref{rem:two-adjoints-sasaki}.

Now put
\[
M:=\Phi(\mathscr{T}(\mathfrak K)).
\]
We claim that $M\subseteq \mathscr{T}\bigl(\mathbf{Lin}(\widetilde{\mathfrak K})\bigr)$ and in fact
\[
M=\mathscr{T}\bigl(\mathbf{Lin}(\widetilde{\mathfrak K})\bigr).
\]

First, every $u\in\widetilde{\mathfrak K}$ is self-adjoint in ${\mathfrak K}$:
indeed $u={\sim} x$ for some $x$, hence $u^*=({\sim} x)^*={\sim} x=u$ by \textup{(\textbf{IDA4})}.
Consequently, $\mathscr{T}(\mathfrak K)$ is simply the (involutive) submonoid generated by
$\widetilde{\mathfrak K}$, i.e.,\ every $k\in\mathscr{T}(\mathfrak K)$ is a finite product
$u_1\odot\cdots\odot u_n$ with $u_i\in\widetilde{ K}$.

Since $\Phi$ is a monoid homomorphism and $\Phi(u_i)=\pi_{u_i}$, we get
\[
\Phi(k)=\pi_{u_1}\circ\cdots\circ\pi_{u_n}.
\]
Because $\mathbf{Lin}(\widetilde{\mathfrak K})$ is closed under composition and contains all
$\pi_u$ ($u\in\widetilde{K}$), it follows that
\[
\Phi(k)\in \mathbf{Lin}(\widetilde{\mathfrak K}),
\qquad
\text{and hence } M\subseteq \mathbf{Lin}(\widetilde{\mathfrak K}).
\]

Moreover, $M$ is an involutive submonoid of $\mathbf{Lin}(\widetilde{\mathfrak K})$ containing
$\widetilde{\mathbf{Lin}(\widetilde{\mathfrak K})}=\Phi(\widetilde{\mathfrak K})$,
so by minimality of $\mathscr{T}\bigl(\mathbf{Lin}(\widetilde{\mathfrak K})\bigr)$ we have
\[
\mathscr{T}\bigl(\mathbf{Lin}(\widetilde{\mathfrak K})\bigr)\subseteq M.
\]
Conversely, every element of $M$ is a finite composition of elements of
$\widetilde{\mathbf{Lin}(\widetilde{\mathfrak K})}$, hence belongs to the involutive submonoid
generated by $\widetilde{\mathbf{Lin}(\widetilde{\mathfrak K})}$, i.e.,
\[
M\subseteq \mathscr{T}\bigl(\mathbf{Lin}(\widetilde{\mathfrak K})\bigr).
\]
Thus $M=\mathscr{T}\bigl(\mathbf{Lin}(\widetilde{\mathfrak K})\bigr)$.

Therefore the restriction
\[
\nu_{{\mathfrak K}}:=\Phi|_{\mathscr{T}(\mathfrak K)}
:\mathscr{T}(\mathfrak K)\to \mathscr{T}\bigl(\operatorname{Lin}(\widetilde{\mathfrak K})\bigr),
\qquad
\nu_{\mathfrak K}(k):=k\bullet(-),
\]
is well-defined and \emph{surjective}.

\smallskip
\noindent\textbf{Step 3: $\nu_{\mathfrak K}$ preserves involution.}
Let $k\in\mathscr{T}(\mathfrak K)$ and write $k=u_1\odot\cdots\odot u_n$ with
$u_i\in\widetilde{K}$. Then $k^*=u_n\odot\cdots\odot u_1$, and hence
\[
\nu_{{\mathfrak K}}(k^*)=\Phi(k^*)=\pi_{u_n}\circ\cdots\circ\pi_{u_1}.
\]
Moreover, each $\pi_u$ is self-adjoint in $\operatorname{Lin}(\widetilde{\mathfrak K})$,
i.e.\ $\pi_u^*=\pi_u$, and the involution on $\operatorname{Lin}(\widetilde{\mathfrak K})$
satisfies $(f\circ g)^*=g^*\circ f^*$. Therefore
\[
\nu_{{\mathfrak K}}(k)^*
=(\pi_{u_1}\circ\cdots\circ\pi_{u_n})^*
=\pi_{u_n}\circ\cdots\circ\pi_{u_1}
=\nu_{{\mathfrak K}}(k^*).
\]
Thus $\nu_{\mathfrak K}$ is a morphism in $\mathbb{IM}$.

\smallskip
\noindent\textbf{Step 4: injectivity of $\nu_{\mathfrak K}$.}
Let $s,t\in\mathscr{T}(\mathfrak K)$ and assume $\nu_{\mathfrak K}(s)=\nu_{\mathfrak K}(t)$. Then
$s\bullet w=t\bullet w$ for all $w\in\widetilde{\mathfrak K}$, i.e.\ $s\equiv t$.
By the standing separation assumption we conclude $s=t$. Hence 
$\nu_{\mathfrak K}$ is injective,
and therefore an isomorphism in $\mathbb{IM}$.

\smallskip
\noindent\textbf{Step 5: restriction to tests and order preservation.}
For each $u\in\widetilde{ K}$ we have 
$\nu_{\mathfrak K}(u)=\pi_u$, so the restriction
\[
\nu_{\mathfrak K}|_{\widetilde{\mathfrak K}}
:\widetilde{\mathfrak K}\to \widetilde{\mathbf{Lin}(\widetilde{\mathfrak K})},
\qquad
u\mapsto \pi_u,
\]
is a bijection.
Let $\le$ denote the ortholattice order on $\widetilde{\mathfrak K}$, and let $\le$ also denote
the distinguished test order on $\widetilde{\mathbf{Lin}(\widetilde{\mathfrak K})}$, given by
\[
p\le q \quad\Longleftrightarrow\quad p=q\circ p.
\]
Then for $u,v\in\widetilde{K}$,
\[
u\le v
\quad\Longleftrightarrow\quad
\pi_v\circ\pi_u=\pi_u
\quad\Longleftrightarrow\quad
\pi_u\le \pi_v.
\]
Indeed, if $u\le v$ then $\pi_u(x)\le u\le v$ for all $x$, hence
$\pi_v(\pi_u(x))=\pi_u(x)$.
Conversely, if $\pi_v\circ\pi_u=\pi_u$, then evaluating at $1$ yields
\[
u=\pi_u(1)=\pi_v(\pi_u(1))=\pi_v(u)\le v.
\]
Thus $\nu|_{\widetilde{\mathfrak K}}$ is order-preserving, completing the proof of
\textup{({\ref{item:T-basic-2}})}.

\medskip

\noindent\textup{({\ref{item:T-basic-3}}):}
Let $\mathcal{M}$ be a complete orthomodular lattice.
Consider the Foulis quantale $\mathbf{Lin}(\mathcal{M})$ and the semi-Foulis dynamic algebra
\[
\mathscr{P}\bigl(\mathscr{T}(\mathbf{Lin}(\mathcal{M}))\bigr).
\]
By Definition~\ref{def:functor-T},
$\mathscr{T}\bigl(\mathscr{P}(\mathscr{T}(\mathbf{Lin}(\mathcal{M})))\bigr)$
is the involutive submonoid generated by the distinguished test set
$\widetilde{\mathscr{P}(\mathscr{T}(\mathbf{Lin}(\mathcal{M})))}$.
By Proposition~\ref{prop:chi-iso-finitary},
\[
\widetilde{\mathscr{P}(\mathscr{T}(\mathbf{Lin}(\mathcal{M})))}
=\bigl\{\{\pi_m\}\mid m\in\mathcal{M}\bigr\}.
\]
In $\mathscr{P}(\mathscr{T}(\mathbf{Lin}(\mathcal{M})))$ we have
\[
\{f\}\odot\{g\}=\{f\odot g\},
\qquad
\{f\}^*=\{f^*\},
\qquad
e=\{\operatorname{id}_{\mathcal{M}}\}.
\]
Define
\[
\mu_{\mathcal{M}}:\mathscr{T}(\mathbf{Lin}(\mathcal{M}))
\longrightarrow
\mathscr{T}\bigl(\mathscr{P}(\mathscr{T}(\mathbf{Lin}(\mathcal{M})))\bigr),
\qquad
\mu_{\mathcal{M}}(f):=\{f\}.
\]
Then $\mu_{\mathcal{M}}$ preserves unit, product, and involution, hence $\mu$ is a morphism in $\mathbb{IM}$.

Every element of $\mathscr{T}\bigl(\mathscr{P}(\mathscr{T}(\mathbf{Lin}(\mathcal{M})))\bigr)$
is a finite product of generators $\{\pi_m\}$, and such a product is always a singleton:
\[
\{\pi_{m_1}\}\odot\cdots\odot\{\pi_{m_n}\}
=\{\pi_{m_1}\odot\cdots\odot\pi_{m_n}\}.
\]
Hence every element in the codomain is of the form $\{f\}$ with
$f\in\mathscr{T}(\mathbf{Lin}(\mathcal{M}))$, so $\mu_{\mathcal{M}}$ is surjective.
Injectivity is immediate: $\mu_{\mathcal{M}}(f)=\mu_{\mathcal{M}}(g)$ implies $\{f\}=\{g\}$ and hence $f=g$.
Thus $\mu_{\mathcal{M}}$ is an isomorphism in $\mathbb{IM}$, proving \textup{({\ref{item:T-basic-3}})}.

\medskip

\noindent\textup{({\ref{item:T-basic-4}}):}
Let $f:\mathfrak{K}_1\to\mathfrak{K}_2$ be a morphism in $\mathbb{IDA}$, with underlying map
$f:K_1\to K_2$.
If $u\in\widetilde{K}_1$, then $u={\sim} x$ for some $x\in K_1$, hence
\[
f(u)=f({\sim} x)={\sim} f(x)\in\widetilde{\mathfrak K}_2,
\]
so $f(\widetilde{\mathfrak K}_1)\subseteq \widetilde{\mathfrak K}_2$.
Since $f$ preserves $\odot$, ${}^*$ and $e$, it follows that $f$ maps the involutive submonoid
generated by $\widetilde{\mathfrak K}_1$ into the involutive submonoid generated by
$\widetilde{\mathfrak K}_2$, i.e.
\[
f\bigl(\mathscr{T}(\mathfrak{K}_1)\bigr)\subseteq \mathscr{T}(\mathfrak{K}_2).
\]
Therefore
\[
\mathscr{T}(f):=f|_{\mathscr{T}(\mathfrak{K}_1)}
:\mathscr{T}(\mathfrak{K}_1)\to\mathscr{T}(\mathfrak{K}_2)
\]
is well defined, and since it is the restriction of an involutive monoid homomorphism,
it is again an involutive monoid homomorphism. This proves \textup{({\ref{item:T-basic-4}})}.
\end{proof}

\section{Construction of $\mathscr{T}$-based orthomodular dynamic algebras from complete orthomodular lattices}
\label{sec:construction-gamma}

The purpose of this section is to detail the construction of 
a functor $\Gamma$ from the category $\mathbb{COL}$ of complete orthomodular lattices to the category $\mathscr{T}\mathbb{ODA}$ of $\mathscr{T}$-based orthomodular dynamic algebras.

\subsection{Mapping of Objects} 
Let $\mathcal M=(M,\le,{}^\perp)$ be a complete orthomodular lattice. We put $L_{\mathcal M}=\mathscr T(\mathbf{Lin}(\mathcal M))$ and $\Gamma(\mathcal M)=\mathscr{P}( L_\mathcal{M})$.
\medskip

Through the following lemma, we establish the correctness of the definition of $\Gamma$ on objects.

\medskip

 \begin{lemma}\label{PMisTODA}
     Let $\mathcal M=(M,\le,{}^\perp)$ be a complete orthomodular lattice. 
     Then the structure $\bigl(\mathscr{P}( L_\mathcal{M}) ,\bigcup,\odot, ^ *, {\sim}, \{\mathrm{id}_M\})$ from the previous construction is a $\mathscr{T}$-based orthomodular dynamic algebra.
 \end{lemma}
  \begin{proof} It is enough to check conditions {(\textbf{TODA1})}-{(\textbf{TODA4})}. 
  
\medskip
\noindent{(\textbf{TODA1})}: It follows from Proposition 
\ref{prop:chi-iso-finitary}.
\medskip

\noindent{(\textbf{TODA2})}: Assume that a set $A$ satisfies  
$\mathscr{T} \left( \mathscr{P}( L_\mathcal{M}) \right) \subseteq A \subseteq \mathscr{P}( L_\mathcal{M})$, $A$ is closed under both the $\odot$ and ${^*}$ operations, and $A$ is closed under $\bigcup$. 
Since $\mathscr{T} \left( \mathscr{P}( L_\mathcal{M}) \right)=%
\{\{x\}\mid x\in L_\mathcal{M}\}$ and $A$ is closed under $\bigcup$ we conclude that  $A=\mathscr{P}( L_{M})$.

\medskip

\noindent{(\textbf{TODA3})}: Since the join in $\mathscr{P}(L)$ is a set-theoretic union and the elements of $\mathscr{P}(L)$, being atoms, are disjoint from one another, (\textbf{TODA3}) is automatically satisfied.

\medskip

\noindent{(\textbf{TODA4})}: Since $\equiv$ is an equivalence 
we obtain that $=\subseteq \equiv$. 

Assume now that $S, T\in %
\mathscr{T}\left(\mathscr{P}\left(\mathscr{T}(\mathbf{Lin}(\mathcal M))\right)\right)$ and $S \equiv T$. Then $S=\{s\}$, $s\in \mathscr{T}(\mathbf{Lin}(\mathcal M))$, $T=\{t\}$, $t\in \mathscr{T}(\mathbf{Lin}(\mathcal M))$ and 
${\sim}\,{\sim} (\{s\}\odot Z)={\sim}\,{\sim} (\{t\}\odot Z)$ for every $Z\in \widetilde{\mathscr{P}\left(\mathscr{T}(\mathbf{Lin}(\mathcal M))\right)}$. 


Considering each singleton $Z = \{m\}$ with $m \in M$, we can compute:
\begin{align*}
  {\sim}\,{\sim} (\{s\}\odot Z)&=   {\sim}\,{\sim} (\{s\}\odot \{\pi_m\})={\sim}\,{\sim} (\{s\circ \pi_m\})%
  \stackrel{L. \ref{lemma4}}{=}\{\pi_{s(\pi_m(1))}\}= \{\pi_{s(m)}\}.
\end{align*}
Hence, $\{\pi_{s(m)}\}=\{\pi_{t(m)}\}$ for every $m\in M$. 
We conclude that $s=t$, i.e., $S=T$. 
   \end{proof}
\medskip

 \subsection{Mapping of Arrows}
 
We formally define the action of the map $\Gamma$ on ortholattice isomorphisms between complete orthomodular lattices. Let $\mathcal{M}_1=(M_{1},\le_{1}, -^{\perp_{1}})$ and $\mathcal{M}_2=(M_{2}, \le_{2}, -^{\perp_{2}})$ be complete orthomodular lattices. The map $\Gamma$ maps these lattices to their respective $\mathscr{T}$-based orthomodular dynamic algebras %
$\Gamma (\mathcal{M}_1) =\bigl(\mathscr{P}( L_{{M}_1}) ,\bigcup_1,\odot_1, ^{*_1}, {\sim_1}, \{\mathrm{id}_{M_1}\} \bigr)$ and $\Gamma (\mathcal{M}_2) =\bigl(\mathscr{P}( L_{{M}_2}) ,\bigcup_2,\odot_2, ^{*_2}, {\sim_2}, \{\mathrm{id}_{M_2}\}\bigr)$.

\medskip

\noindent{}For an ortholattice isomorphism $k \colon \mathcal{M}_1 \to \mathcal{M}_2$, we define a map $\Gamma(k) \colon \Gamma (\mathcal{M}_1) \to \Gamma (\mathcal{M}_2)$ as follows:

$$A \mapsto \{k \circ a \circ k^{-1} \mid a \in A\}, A\in \mathscr{P}( L_{{M}_1}).$$
Clearly, $\Gamma(k)(A)\in \mathscr{P}( L_{{M}_2})$,  and the inverse map to 
$\Gamma(k)$ is $\Gamma(k^{-1})$. 

\noindent Suppose that $m \in M_1$. Then $\pi_{\,k(m)}=k\circ \pi_m\circ k^{-1}$. Namely,  for every $n \in M_1$:

\[
\begin{aligned}
{\pi _{k\left( m \right)}}\left( n \right) &= k\left( m \right) \wedge \left( {{{\left( {k\left( m \right)} \right)}^ \bot } \vee n} \right)%
= k\left( {m \wedge \left( {{m^ \bot } \vee {k^{ - 1}}\left( n \right)} \right)} \right)\\
&= k\left( {{\pi _m}\left( {{k^{ - 1}}\left( n \right)} \right)} \right)%
=\left( {k \circ {\pi _m} \circ {k^{ - 1}}} \right)\left( n \right).
\end{aligned}
\]

\medskip

The following lemma verifies the definition of $\Gamma$ on morphisms.

\medskip

\begin{lemma}\label{gammamor}
Let $\mathcal M_{1}=(M_{1},\le_{1}, -^{\perp_{1}})$ and $\mathcal M_{2}=(M_{2}, \le_{2}, -^{\perp_{2}})$ be complete orthomodular lattices. If $k: \mathcal M_{1} \to \mathcal M_{2}$ is an ortholattice isomorphism then the map $\Gamma(k)$ is a $\mathscr{T}\mathbb{ODA}$-morphism.
\end{lemma}
\begin{proof} Clearly, $\Gamma(k)$ is a bijective map preserving arbitrary joins that are unions. Assume that $A, B\in  \mathscr{P}( L_{{M}_1})$. We compute:
\begin{align*}
    \Gamma(k)(A)\odot_2 \Gamma(k)(B)&=%
    \{k \circ a \circ k^{-1} \mid a \in A\} \odot_2 \{k \circ b \circ k^{-1} \mid b \in B\}\\
    &=\{k \circ a \circ k^{-1} \circ k \circ b \circ k^{-1} \mid a \in A, b \in B\} \\
    &=\{k \circ a \circ b \circ k^{-1} \mid a \in A, b \in B\} \\
    &=\{k \circ c \circ k^{-1} \mid c \in A\odot_1 B\}=%
    \Gamma(k)(A\odot_1 B),
\end{align*}
Since the composition of linear maps is linear, the adjoint of $k$ exists and $k^{*}=k^{-1}$ we can compute:
\begin{align*}
     \Gamma(k)(A)^{*_2}&=\{k \circ a \circ k^{-1} \mid a \in A\}^{*_2}%
     =\{\left(k \circ a \circ k^{-1}\right)^{*_2} \mid a \in A\}\\
     &=\{(k^{-1})^{*} \circ a^{*_1} \circ k^{*} \mid a \in A\}=%
     \{k \circ a^{*_1} \circ k^{-1} \mid a^{*_1} \in A^{*_1}\}=\Gamma(k)(A^{*_1}).\\
\end{align*}
Since $k$ preserves arbitrary joins and orthocomplement, and $k^{-1}(1)=1$ we obtain:
\begin{align*}
     {\sim_2}\Gamma(k)(A)&={\sim_2}\{k \circ a \circ k^{-1} \mid a \in A\}=\{\pi_{\left(\bigvee_{a\in A} \left(k \circ a \circ k^{-1}\right)(1)\right)^{\perp_2}}\}\\
     &=\{\pi_{\left(k\left(\bigvee_{a\in A}  a(1)\right)\right)^{\perp_2}}\}%
     =\{\pi_{k\left(\left(\bigvee_{a\in A}  a(1)\right)^{\perp_1}\right)}\}\\%
      &=\{k\circ \pi_{\left(\bigvee_{a\in A}  a(1)\right)^{\perp_1}} \circ k^{-1}\}=\Gamma(k)({\sim_1} A).
\end{align*} 

We compute
\begin{align*}
    \Gamma(k)(\{\mathrm{id}_{M_1}\})&=\{k\circ \mathrm{id}_{M_1}\circ k^{-1}\}=\{\mathrm{id}_{M_2}\}. 
\end{align*}
We conclude that $\Gamma(k)$ is a $\mathscr{T}\mathbb{ODA}$-morphism.
\end{proof}

\medskip
The following lemma demonstrates that the functor $\Gamma$ preserves the composition of morphisms.

\medskip

\begin{lemma}\label{comp}
Let $\mathcal{M}_1$, $\mathcal{M}_2$, and $\mathcal{M}_3$ be complete orthomodular lattices. Let $k \colon \mathcal{M}_1 \to \mathcal{M}_2$ and $l \colon \mathcal{M}_2 \to \mathcal{M}_3$ be $\mathbb{COL}$-morphisms. Then, the functor $\Gamma$ preserves the composition of morphisms, i.e., $\Gamma(l \circ k) = \Gamma(l) \circ \Gamma(k)$.
\end{lemma}

\begin{proof}
Let $k: \mathcal{M}_1 \to \mathcal{M}_2$ and $l: \mathcal{M}_2 \to \mathcal{M}_3$ be arbitrary $\mathbb{COL}$-morphisms. We aim to demonstrate that $\Gamma(l \circ k) = \Gamma(l) \circ \Gamma(k)$.

For any $A_1 \in \mathscr{P}(L_{M_1})$, applying the definition of $\Gamma$ to the composite morphism $l \circ k$ yields:
\[ \Gamma(l \circ k)(A_1) = \{(l \circ k) \circ a \circ (l \circ k)^{-1} \mid a \in A_1\}. \]

Next, we evaluate the composition $\Gamma(l) \circ \Gamma(k)$ by applying $\Gamma(l)$ to the result of $\Gamma(k)(A_1)$:
\begin{align*} (\Gamma(l) \circ \Gamma(k))(A_1) &= \Gamma(l)(\Gamma(k)(A_1)) \\ &= \Gamma(l)(\{k \circ a \circ k^{-1} \mid a \in A_1\}) \\ &= \{l \circ (k \circ a \circ k^{-1}) \circ l^{-1} \mid a \in A_1\} \end{align*}

By the associativity of function composition and the property of inverse functions, $(f \circ g)^{-1} = g^{-1} \circ f^{-1}$, we can rewrite the expression as:
\[ l \circ (k \circ a \circ k^{-1}) \circ l^{-1} = (l \circ k) \circ a \circ (k^{-1} \circ l^{-1}) = (l \circ k) \circ a \circ (l \circ k)^{-1}. \]

Substituting this back into the expression for $(\Gamma(l) \circ \Gamma(k))(A_1)$, we obtain:
\[ (\Gamma(l) \circ \Gamma(k))(A_1) = \{(l \circ k) \circ a \circ (l \circ k)^{-1} \mid a \in A_1\} \]
Comparing this result with the definition of $\Gamma(l \circ k)(A_1)$, we conclude that:
\[ \Gamma(l \circ k)(A_1) = (\Gamma(l) \circ \Gamma(k))(A_1) \]

Since this equality holds for all $A_1 \in \mathscr{P}(L_{M_1})$, we conclude that $\Gamma(l \circ k) = \Gamma(l) \circ \Gamma(k)$.
\end{proof}

We now show that $\Gamma$ preserves identity morphisms. Let $\mathrm{id}_{\mathcal{M}}$ denote the identity morphism on $\mathcal{M}$ in the category $\mathbb{COL}$ of complete orthomodular lattices, and $\mathrm{id}_{\mathcal{K}}$ the identity morphism on $\mathcal{K}$ in the category of 
$\mathscr{T}$-based orthomodular dynamic algebras.

 \medskip

\begin{lemma}\label{id}
Let $\mathcal{M}$ be a complete orthomodular lattice. 
The map $\Gamma$ preserves the identity morphism on $\mathcal{M}$, that is, $\Gamma(\text{\rm id}_{\mathcal{M}}) = \text{\rm id}_{\Gamma(\mathcal{M})}$.
\end{lemma}

\begin{proof}
Let $\mathcal{M} = (M, \le, -^{\perp})$ be a complete orthomodular lattice. Our objective is to demonstrate that $\Gamma(\text{id}_{\mathcal{M}}) = \text{id}_{\Gamma(\mathcal{M})}$.

By the definition of $\Gamma$, for any $A \in \mathscr{P}(L_M)$:
\[ \Gamma(\mathrm{id}_{\mathcal{M}})(A) = \{\mathrm{id}_{\mathcal{M}} \circ a \circ \mathrm{id}_{\mathcal{M}}^{-1} \mid a \in A\}. \]

Since $\mathrm{id}_{\mathcal{M}}$ is the identity morphism, for any $a \in A$, we have $\mathrm{id}_{\mathcal{M}} \circ a = a$ and $a \circ \mathrm{id}_{\mathcal{M}}^{-1} = a$. Thus, the expression simplifies to:
\[ \Gamma(\mathrm{id}_{\mathcal{M}})(A) = \{a \mid a \in A\} = A. \]

Since this equality holds for every $A \in \mathscr{P}( L_{{M}})$, it follows directly from the definition of the identity morphism on $\Gamma(\mathcal{M})$ that $\Gamma(\text{id}_{\mathcal{M}})$ acts as the identity on $\Gamma(\mathcal{M})$.

Therefore, $\Gamma(\mathrm{id}_{\mathcal{M}}) = \mathrm{id}_{\Gamma(\mathcal{M})}$.
\end{proof}

Thus, $\Gamma$ is a functor from the category $\mathbb{COL}$ of complete orthomodular lattices to the category $\mathscr{T}\mathbb{ODA}$ of $\mathscr{T}$-based orthomodular dynamic algebras.

\medskip

\begin{theorem}\label{Func COL to TODA}
The map $\Gamma$ constitutes a functor from the category of complete orthomodular lattices, $\mathbb{COL}$, to the category of $\mathscr{T}$-based orthomodular dynamic algebras, $\mathbb{\mathscr{T}ODA}$.
\end{theorem}
\begin{proof}
The assertion follows directly from Lemmas \ref{PMisTODA}, \ref{gammamor}, \ref{comp} and \ref{id}.
\end{proof}

\section{The Functor $\Psi $ from  $\mathscr{T}$-based orthomodular dynamic algebras to complete orthomodular lattices} 
\label{sec:functor-psi}

This section defines a functor $\Psi \colon \mathscr{T}\mathbb{ODA} \to \mathbb{COL}$. The functor $\Psi$ assigns a complete orthomodular lattice to each $\mathscr{T}$-based orthomodular dynamic algebra and an ortholattice isomorphism to each $\mathscr{T}\mathbb{ODA}$-morphism.

\subsection{Mapping of Objects}

Let $\mathfrak{K} = (K, \bigsqcup, \odot, {}^{*}, \sim, e)$ be a $\mathscr{T}$-based orthomodular dynamic algebra. The object mapping of $\Psi$ is defined as $\Psi(\mathfrak{K}) = \widetilde{\mathfrak{K}} = (\widehat{K}, \preceq, \perp)$. By Definition \ref{finitary-goda}, $\Psi(\mathfrak{K})$ is a complete orthomodular lattice.


\subsection{Mapping of Arrows}

Let $\mathfrak{K}_1 = (K_1, \bigsqcup_1, \odot_1, {}^{*}_1, \sim_1, e_1)$ and $\mathfrak{K}_2 = (K_2, \bigsqcup_2, \odot_2, {}^{*}_2, \sim_2, e_2)$ be $\mathscr{T}$-based orthomodular dynamic algebras, and let $\phi \colon \mathfrak{K}_1 \to \mathfrak{K}_2$ be a $\mathscr{T}\mathbb{ODA}$-morphism.

We define $\Psi(\phi) = \widetilde{\phi} \colon \Psi(\mathfrak{K}_1) \to \Psi(\mathfrak{K}_2)$. By Lemma \ref{morisorthoiso}, $\Psi(\phi)$ is a $\mathbb{COL}$-morphism.

\medskip
The following theorem establishes that $\Psi$ is a functor from the category $\mathscr{T}\mathbb{ODA}$ of $\mathscr{T}$-based orthomodular dynamic algebras to the category $\mathbb{COL}$ of complete orthomodular lattices. This means $\Psi$ preserves both objects and the structure of morphisms between these categories.

\medskip

\begin{theorem}\label{thm:G-is-functor}
The map $\Psi$ constitutes a functor from the category of $\mathscr{T}$-based orthomodular dynamic algebras ($\mathscr{T}\mathbb{ODA}$) to the category of complete orthomodular lattices ($\mathbb{COL}$).
\end{theorem}

\begin{proof}
Let $\mathfrak{K} = (K, \bigsqcup, \odot, {}^{*}, \sim, e)$ be a $\mathscr{T}$-based orthomodular dynamic algebra, i.e., an object in $\mathscr{T}\mathbb{ODA}$.

\noindent \textit{Preservation of Identity Morphisms}:
Consider the identity $\mathscr{T}\mathbb{ODA}$-morphism $\mathrm{id}_{\mathfrak{K}}: \mathfrak{K} \to \mathfrak{K}$. By the definition of $\Psi$, for any $v \in \widetilde K$ (where $\widetilde K$ is the underlying set of $\Psi(\mathfrak{K})$), we have:
$$ \Psi(\mathrm{id}_\mathfrak{K})(v) = \mathrm{id}_{\widetilde{\mathfrak{K}}}(v) = v.$$

This implies that $\Psi(\mathrm{id}_\mathfrak{K})$ acts as the identity mapping on $\widetilde K$, which is precisely $\mathrm{id}_{\Psi(\mathfrak{K})}$. Hence, $\Psi(\mathrm{id}_\mathfrak{K})=\mathrm{id}_{\Psi(\mathfrak{K})}$.

\noindent \textit{Preservation of Composition of Morphisms}:
Let $\mathfrak{K}_1$, $\mathfrak{K}_2$, and $\mathfrak{K}_3$ be $\mathscr{T}$-based orthomodular dynamic algebras, and let $\phi: \mathfrak{K}_1 \to \mathfrak{K}_2$ and $\varphi:\mathfrak{K}_2 \to \mathfrak{K}_3$ be $\mathscr{T}\mathbb{ODA}$-morphisms. For any $v \in \widetilde K_1$:
\begin{align*}
\Psi(\varphi\circ\phi)(v) &= (\varphi\circ\phi)(v) && \text{(by definition of } \Psi(\text{morphism})) \\
&= \varphi\bigl(\phi(v)\bigr) && \text{(by definition of function composition)} \\
&= \Psi(\varphi)\bigl(\phi(v)\bigr) && \text{(by definition of } \Psi(\text{morphism})) \\
&= \Psi(\varphi)\bigl(\Psi(\phi)(v)\bigr) && \text{(by definition of } \Psi(\text{morphism})) \\
&= \bigl(\Psi(\varphi)\circ \Psi(\phi)\bigr)(v) && \text{(by definition of function composition)}
\end{align*}
Since this equality holds for all $v \in \widetilde K_1$, we conclude that $\Psi(\varphi\circ\phi)=\Psi(\varphi)\circ \Psi(\phi)$.

Since $\Psi$ preserves identity morphisms and the composition of morphisms, it is a functor from $\mathscr{T}\mathbb{ODA}$ to $\mathbb{COL}$.
\end{proof}

\subsection{The Natural Isomorphism $\mu : 1_{\mathbb{COL}} \stackrel{\cong}{\Rightarrow} \Psi\circ \Gamma$}

From the definition of the functors $\Psi$ and $\Gamma$, for any orthomodular lattice $\mathcal{M}$, we have:
\[ \left( {\Psi \circ \Gamma } \right)\left( \mathcal{M} \right) = \Psi \left( {\Gamma \left( \mathcal{M} \right)} \right) = \Psi \left( \mathscr{P}(L_\mathcal{M}) \right) = \widetilde {\mathscr{P}(L_\mathcal{M})} \]

Our goal is to establish a natural isomorphism between the identity functor $1_{\mathbb{COL}}$ and the composite functor $\Psi \circ \Gamma$.

\medskip

\begin{theorem}\label{thm:tau-natural}
Let ${\mathcal{M}} = \left( {{M},{ \le },{}^{{ \bot }} } \right)$ be a complete orthomodular lattice, and let $\Gamma (\mathcal {M}) =\bigl(\mathscr{P}(L_\mathcal{M}) ,\bigcup,\odot,{\sim}, { - ^*}\bigr)$ be its corresponding $\mathscr{T}$-based orthomodular dynamic algebra. Furthermore, let $\delta : \mathcal{M} \to {\widetilde {\mathscr{P}(L_\mathcal{M})}}$ be the established ortholattice isomorphism, as defined in Theorem~\ref{prop:chi-iso-finitary}. Define a natural transformation $\mu : 1_{\mathbb{COL}} \Rightarrow \Psi\circ \Gamma$ such that for every orthomodular lattice $\mathcal M$, the component $\mu_{\mathcal M}$ is precisely $\delta$. Thus, $\mu$ is a natural isomorphism.
\end{theorem}
\begin{proof}
From Proposition~\ref{prop:chi-iso-finitary}, we know that for each orthomodular lattice $\mathcal{M}$, the component $\mu_{\mathcal M} = \delta_{\mathcal M}$ is an $\mathbb{COL}$-isomorphism, thus ensuring its bijectivity.

To demonstrate that $\mu$ is a natural transformation, we must verify the commutativity of the following diagram for any $\mathbb{COL}$-morphism $k : \mathcal M_1 \to \mathcal M_2$:
\[
\begin{tikzcd}
\mathcal M_1 \ar[r,"k"] \ar[d,"\mu_{\mathcal M_1}"']
& \mathcal M_2 \ar[d,"\mu_{\mathcal M_2}"]\\
\Psi\bigl(\Gamma(\mathcal M_1)\bigr)\ar[r,"\Psi(\Gamma(k))"']
& \Psi\bigl(\Gamma(\mathcal M_2)\bigr)
\end{tikzcd}
\]
For any element $m \in M_1$, we evaluate both compositions:

1.  Path through $\mu_{\mathcal M_2} \circ k$ (by the definition of $\mu_{\mathcal M_2}$, which is $\delta_{\mathcal M_2}$):
    $$ (\mu_{\mathcal M_2}\circ k)(m) = \mu_{\mathcal M_2}(k(m)) = \{\pi_{k(m)}\} $$

2.  Path through $\Psi(\Gamma(k)) \circ \mu_{\mathcal M_1}$ (first by the definition of $\mu_{\mathcal M_1}$, which is $\delta_{\mathcal M_1}$, and then by the definition of the functors $\Psi$ and $\Gamma$):
\begin{align*}
    (\Psi(\Gamma(k))\circ\mu_{\mathcal M_1})(m) &= \Psi(\Gamma(k))(\mu_{\mathcal M_1}(m))=\Psi(\Gamma(k))(\{\pi_m\})= \Gamma(k)(\{\pi_m\})\\
    &=\{k \circ \pi_m \circ k^{-1}\}= \{\pi_{k(m)}\}.
\end{align*}

Comparing the results from both paths, we obtain that 
   $\Psi(\Gamma(k))\circ\mu_{\mathcal M_1}=\mu_{\mathcal M_2}\circ k$.

Given that each component $\mu_{\mathcal M}$ is an isomorphism and the naturality condition is satisfied, we conclude that $\mu$ is a natural isomorphism.
\end{proof}

\subsection{The Natural Isomorphism $\lambda : 1_{\mathscr{T}\mathbb{ODA}} \stackrel{\cong}{\Rightarrow} \Gamma \circ \Psi$}
\label{section:natural_iso_lambda}

Let $\mathfrak{K} = (K , {\bigsqcup } , {\odot } , {}{^*}, {\sim}, e)$   be a $\mathscr{T}$-based orthomodular dynamic algebra. Pursuant to Definition \ref{finitary-goda}, the associated structure $(\widetilde K , {\preceq } , {}^{\perp} )$ is a complete orthomodular lattice.

 By the definitions of the functors $\Psi$ and $\Gamma$, the action of the composite functor on $\mathfrak{K}$ yields the following algebraic structure:
    $$(\Gamma \circ \Psi)(\mathfrak{K}) = \Gamma(\Psi(\mathfrak{K})) = \Gamma(\widetilde {\mathfrak{K}}) = 
    \big(\mathscr{P}( \mathscr T(\mathbf{Lin}(\widetilde{\mathfrak{K}}))),%
    \bigcup,\odot, ^ *, {\sim}, \{\mathrm{id}_M\}\big).$$

     Lemma~\ref{lemma2} guarantees that every element $k \in K$ possesses a unique representation as a  join  of elements from $\mathscr T({\mathfrak{K}})$.

\medskip 

     We formally define the natural transformation $\lambda : 1_{\mathscr{T}\mathbb{ODA}} \Rightarrow \Gamma\circ\Psi$ by specifying its component map $\lambda_\mathfrak{K} : K \to \mathscr{P}( \mathscr T(\mathbf{Lin}(\widetilde{\mathfrak{K}})))$ for each $\mathfrak{K} \in \operatorname{Ob}(\mathscr{T}\mathbb{ODA})$. For an element $k \in K$ with its unique representation:
    $$k = \bigsqcup S_k, \ S_k\subseteq \mathscr T({\mathfrak{K}})$$
    the component map $\lambda_\mathfrak{K}$ is defined as:
    $${\lambda_{\mathfrak{K}}}(k) = \left\{\nu_{\mathfrak{K}}(s)\mid s\in S_k \right\}\subseteq \mathscr T(\mathbf{Lin}(\widetilde{\mathfrak{K}})).$$

Our subsequent objective is to demonstrate that $\lambda$ is a natural isomorphism.

\medskip 

\begin{theorem}\label{NI Lambda} The natural transformation 
$\lambda : 1_{\mathscr{T}\mathbb{ODA}} \stackrel{\cong}{\Rightarrow} \Gamma\circ\Psi$ is a natural isomorphism.
\end{theorem}

\begin{proof}
The proof consists of two main steps: first, showing  that each component $\lambda_\mathfrak{K}$ is a $\mathscr{T}\mathbb{ODA}$–isomorphism, and second, verifying the naturality condition.

\medskip 

\begin{enumerate}[label=\arabic*., ref=\arabic*]
    \item \label{subsection:isomorphism} \textbf{Component Isomorphism ($\lambda_\mathfrak{K}$ is a $\mathscr{T}\mathbb{ODA}$–Isomorphism)}

    \noindent{}Let $\mathfrak{K}$ be an arbitrary $\mathscr{T}$-based orthomodular dynamic algebra. The map $\lambda_{\mathfrak{K}} \colon K \to (\Gamma \circ \Psi)(\mathfrak{K})$, defined above, is a $\mathscr{T}\mathbb{ODA}$–isomorphism for the following reasons:

    Fix $\mathfrak K=(K,\bigsqcup,\odot,{}^{*},{\sim},e)\in\mathscr{T}\mathbb{ODA}$.
By {\rm({\bfseries TODA1})}, $\widetilde{\mathfrak K}$ is a complete orthomodular lattice, hence
$\mathbf{Lin}(\widetilde{\mathfrak K})$ is a Foulis quantale. In particular,
$\mathscr{T}(\mathbf{Lin}(\widetilde{\mathfrak K}))$ is well-defined.

Since $\mathfrak{K}$ is $\mathscr{T}$-based and satisfies 
{\rm({\bfseries TODA4})}, property ({\bfseries T2}) of functor  $\mathscr{T}$ provides an isomorphism of involutive monoids:
\[
\nu_{\mathfrak K}:\mathscr{T}(\mathfrak K)\longrightarrow
\mathscr{T}\bigl(\mathbf{Lin}(\widetilde{\mathfrak K})\bigr),
\qquad
\nu_{\mathfrak K}(t):=t\bullet(-).
\]
For each $k\in K$, let $S_k\subseteq \mathscr{T}(\mathfrak K)$ be the \emph{unique} set
such that
\begin{equation}\label{eq:Sk-normal-form}
k=\bigsqcup S_k,
\end{equation}
whose existence and uniqueness are guaranteed by Lemma~{\rm 4.3}.
Define
\[
\lambda_{\mathfrak K}:K\longrightarrow
\mathscr{P}\!\left(\mathscr{T}\bigl(\mathbf{Lin}(\widetilde{\mathfrak K})\bigr)\right),
\qquad
\lambda_{\mathfrak K}(k):=\{\nu_{\mathfrak K}(t)\mid t\in S_k\}.
\]

\medskip
\noindent\textbf{Step 1: $\lambda_{\mathfrak K}$ is bijective.}

\smallskip
\noindent\emph{Injectivity.}
Suppose $\lambda_{\mathfrak K}(k)=\lambda_{\mathfrak K}(\ell)$.
Then,
\[
\{\nu_{\mathfrak K}(t)\mid t\in S_k\}=\{\nu_{\mathfrak K}(s)\mid s\in S_\ell\}.
\]
Since $\nu_{\mathfrak K}$ is injective, we get $S_k=S_\ell$.
Hence by \eqref{eq:Sk-normal-form},
\[
k=\bigsqcup S_k=\bigsqcup S_\ell=\ell,
\]
so $\lambda_{\mathfrak K}$ is injective.

\smallskip
\noindent\emph{Surjectivity.}
Let $X\subseteq \mathscr{T}(\mathbf{Lin}(\widetilde{\mathfrak K}))$ be arbitrary.
Since $\nu_{\mathfrak K}$ is bijective, there is a unique subset
$S\subseteq \mathscr{T}(\mathfrak K)$ such that $X=\nu_{\mathfrak K}[S]$.
Let $k=\bigsqcup S\in K$.
By the uniqueness in Lemma~\ref{lemma2}, $S_k=S$, and thus 
\[
\lambda_{\mathfrak K}(k)=\{\nu_{\mathfrak K}(t)\mid t\in S_k\}
=\{\nu_{\mathfrak K}(t)\mid t\in S\}=X.
\]
Thus, $\lambda_{\mathfrak K}$ is surjective.

\medskip
\noindent\textbf{Step 2: $\lambda_{\mathfrak K}$ preserves the operations.}

\medskip 

Recall that $\Gamma(\Psi(\mathfrak K))=
\mathscr{P}(\mathscr{T}(\mathbf{Lin}(\widetilde{\mathfrak K})))$ has: 
\begin{itemize}
    \item Joins given by set union,
    \item Products given by setwise composition:
\[ A \odot B = \{f \circ g \mid f \in A, g \in B\}, \]
\item Involution given setwise:
\[ A^* = \{f^* \mid f \in A\}, \]
\item Unit given by $\{ \mathrm{id}_{\widetilde{\mathfrak{K}}} \}$.
\end{itemize}

\medskip
\noindent\emph{Preservation of arbitrary joins.}
Let $\{k_\alpha\}_{\alpha\in\Lambda}\subseteq K$ and write $k_\alpha=\bigsqcup S_{k_\alpha}$.
Then
\[
\bigsqcup_{\alpha\in\Lambda}k_\alpha
=\bigsqcup_{\alpha\in\Lambda}\bigsqcup S_{k_\alpha}
=\bigsqcup\Bigl(\bigcup_{\alpha\in\Lambda}S_{k_\alpha}\Bigr).
\]
By the uniqueness of the decomposition we have
$S_{\bigsqcup_{\alpha}k_\alpha}=\bigcup_{\alpha}S_{k_\alpha}$.
Therefore, 
\[
\lambda_{\mathfrak K}\Bigl(\bigsqcup_{\alpha\in\Lambda}k_\alpha\Bigr)
=\{\nu_{\mathfrak K}(t)\mid t\in \bigcup_{\alpha}S_{k_\alpha}\}
=\bigcup_{\alpha\in\Lambda}\{\nu_{\mathfrak K}(t)\mid t\in S_{k_\alpha}\}
=\bigcup_{\alpha\in\Lambda}\lambda_{\mathfrak K}(k_\alpha).
\]

\smallskip
\noindent\emph{Preservation of multiplication.}
Let $k,\ell\in K$ with $k=\bigsqcup S_k$ and $\ell=\bigsqcup S_\ell$.
Since $\odot$ distributes over arbitrary joins in each coordinate, we obtain
\[
k\odot \ell
=\Bigl(\bigsqcup S_k\Bigr)\odot\Bigl(\bigsqcup S_\ell\Bigr)
=\bigsqcup\{\,s\odot t\mid s\in S_k,\ t\in S_\ell\,\}.
\]
Thus $S_{k\odot\ell}=\{s\odot t\mid s\in S_k,\ t\in S_\ell\}$ by uniqueness.
Therefore
\begin{align*}
\lambda_{\mathfrak K}(k\odot\ell)
&=\{\nu_{\mathfrak K}(s\odot t)\mid s\in S_k,\ t\in S_\ell\}\\
&=\{\nu_{\mathfrak K}(s)\circ \nu_{\mathfrak K}(t)\mid s\in S_k,\ t\in S_\ell\}
\qquad(\text{$\nu_{\mathfrak K}$ is a monoid homomorphism})\\
&=\lambda_{\mathfrak K}(k)\odot \lambda_{\mathfrak K}(\ell).
\end{align*}

\smallskip
\noindent\emph{Preservation of involution.}
Let $k=\bigsqcup S_k$. Since ${}^*$ preserves arbitrary joins,
\[
k^{*}=\Bigl(\bigsqcup S_k\Bigr)^{*}
=\bigsqcup\{\,s^{*}\mid s\in S_k\,\},
\]
hence $S_{k^{*}}=\{s^{*}\mid s\in S_k\}$ by uniqueness.
Consequently,
\begin{align*}
\lambda_{\mathfrak K}(k^{*})
&=\{\nu_{\mathfrak K}(s^{*})\mid s\in S_k\}
=\{\nu_{\mathfrak K}(s)^{*}\mid s\in S_k\}
=\lambda_{\mathfrak K}(k)^{*},
\end{align*}
since $\nu_{\mathfrak K}$ is a homomorphism of involutive monoids.

\smallskip
\noindent\emph{Preservation of the unit.}
Since $e\in \mathscr{T}(\mathfrak K)$ and $e=\bigsqcup\{e\}$, we have $S_e=\{e\}$. Thus, 
\[
\lambda_{\mathfrak K}(e)=\{\nu_{\mathfrak K}(e)\}=\{id_{\widetilde{\mathfrak K}}\}.
\]

Thus $\lambda_{\mathfrak K}$ is an isomorphism in $\mathscr{T}\mathbb{ODA}$.

\smallskip
\noindent\emph{Preservation of ${\sim}$.}
Let $k\in K$ and write $k=\bigsqcup S_k$.
Since ${\sim} k\in \widetilde{\mathfrak K}\subseteq \mathscr{T}(\mathfrak K)$, it is an atom of $K$
by Lemma~\ref{lemma-atomic}\textup{(\emph{i})}. Hence
\[
S_{{\sim} k}=\{{\sim} k\},
\quad\text{and so}\quad
\lambda_{\mathfrak K}({\sim} k)=\{\nu_{\mathfrak K}({\sim} k)\}.
\]
On the other hand, in $\mathscr{P}\!\bigl(\mathscr{T}(\mathbf{Lin}(\widetilde{\mathfrak K}))\bigr)$, 
the operation ${\sim}$ is given by Construction~\ref{constrpm}\textup{(5)}:
\[
{\sim} A=\Bigl\{\pi_{\bigl(\bigvee_{a\in A} a(1)\bigr)^{\perp}}\Bigr\}.
\]
Applying this to $A=\lambda_{\mathfrak K}(k)=\{\nu_{\mathfrak K}(s)\mid s\in S_k\}$ gives
\[
{\sim}\lambda_{\mathfrak K}(k)
=
\Bigl\{\pi_{\bigl(\bigvee_{s\in S_k}\nu_{\mathfrak K}(s)(1)\bigr)^{\perp}}\Bigr\}.
\]
By \textup{({\bfseries TODA1})}, the top element $1$ of the orthomodular lattice $\widetilde{\mathfrak K}$ is $e$,
so $1=e$ in $\widetilde{\mathfrak K}$. Therefore
\[
\nu_{\mathfrak K}(s)(1)=\nu_{\mathfrak K}(s)(e)=s\bullet e.
\]
By Theorem~\ref{thm:K-module}, $s\bullet e={\sim}{\sim}(s\odot e)={\sim}{\sim} s$, 
since $e$ is the unit of $\odot$. Hence
\[
\bigvee_{s\in S_k}\nu_{\mathfrak K}(s)(1)
=
\bigvee_{s\in S_k}{\sim}{\sim} s
=
{\sim}{\sim}\Bigl(\bigsqcup_{s\in S_k}s\Bigr)
=
{\sim}{\sim} k,
\]
where we used join-preservation of ${\sim}{\sim}$ (Proposition~\ref{prop:K-module}).
Consequently,
\[
{\sim}\lambda_{\mathfrak K}(k)
=
\Bigl\{\pi_{({\sim}{\sim} k)^{\perp}}\Bigr\}
=
\{\pi_{{\sim} k}\}
=
\{\nu_{\mathfrak K}({\sim} k)\}
=
\lambda_{\mathfrak K}({\sim} k),
\]
using $x^{\perp}:={\sim} x$ on $\widetilde{\mathfrak K}$ and
${\sim}({\sim}{\sim} k)={\sim} k$ (Theorem~\ref{thm:K-module}).

\medskip

    \item \label{subsection:naturality} \textbf{Naturality of $\lambda$}

    We now show  that the collection of maps $\lambda = \{\lambda_\mathfrak{K}\}$ satisfies the naturality condition. For any $\mathscr{T}\mathbb{ODA}$–morphism $\phi: \mathfrak{K}_1\to\mathfrak{K}_2$, the following diagram must commute:

    $$
    \begin{tikzcd}
    \mathfrak{K}_1 \ar[r,"{\lambda_{\mathfrak{K}_1}}"] \ar[d,"{\phi}"'] &
    (\Gamma\circ\Psi)(\mathfrak{K}_1) \ar[d,"{(\Gamma\circ\Psi)(\phi)}"] \\
    \mathfrak{K}_2 \ar[r,"{\lambda_{\mathfrak{K}_2}}"'] &
    (\Gamma\circ\Psi)(\mathfrak{K}_2)
    \end{tikzcd}
    $$

    Consider an arbitrary element $k \in K_1$, with its unique representation $k =\bigsqcup S_k$, where $S_k\subseteq  \mathscr T({\mathfrak{K}})$.

    \begin{enumerate}[label=\textbf{Path \arabic*:}]
        \item \label{subsubsection:flow1}  {$\lambda_{\mathfrak{K}_2} \circ \phi$}

        Since $\phi$ is a $\mathscr{T}\mathbb{ODA}$–morphism, it preserves the join and product operations, and maps elements from $\widetilde{K}_1$ to $\widetilde{K}_2$.
        \vspace{0,2cm}
        \begin{center}
            $
        \phi(k) = \phi\!\left(\bigsqcup S_k\right) = %
        \bigsqcup \phi\!\left(S_k\right).
        $
        \end{center}
        \vspace{0,2cm}
        Applying the component map $\lambda_{\mathfrak{K}_2}$ to the result:
        \begin{equation}
        \tag{A}
        (\lambda_{\mathfrak{K}_2})(\phi(k)) = \left(\nu_{\mathfrak{K}_2}\circ \phi\right)\!\left(S_k\right)=%
        \{\phi(s)\bullet_2 (-)\mid s\in S_k\}.
        \end{equation}

        \item \label{subsubsection:flow2} \textbf{$(\Gamma \circ \Psi)(\phi) \circ \lambda_{\mathfrak{K}_1}$}

        First, applying $\lambda_{\mathfrak{K}_1}$ to $k$:
        $$
        (\lambda_{\mathfrak{K}_1})(k) = \nu_{\mathfrak{K}_1}\left(S_k\right)=
         \{s\bullet_1 (-)\mid s\in S_k\}.
        $$
        Next, we apply the functorial action $(\Gamma \circ \Psi)(\phi)$. Recall that $\Psi(\phi)$ is the restriction $\phi|_{\widetilde{K}_1}$, which is an $\mathbb{COL}$-morphism. 
        For every $X\in \mathscr{P}(\mathscr T(\mathbf{Lin}(\widetilde{\mathfrak{K}_1})))$
        $$
        (\Gamma \circ \Psi)(\phi)(X) = \Bigl\{ \Psi(\phi) \circ x \circ \Psi(\phi)^{-1} \;\Big|\; x \in X \Bigr\}
        $$
        For a $z\in {\widetilde{K}_2}$ and $s\in S_k$, we compute:
        \begin{align*}
            \left(\Psi(\phi) \circ (s\bullet_1 (-)) \circ \Psi(\phi)^{-1}\right)(z)&=%
            \left(\Psi(\phi) \circ (s\bullet_1 (-))\right)%
            \left(\Psi(\phi)^{-1}(z)\right)\\
            &=\phi \left(s\bullet_1 (\Psi(\phi)^{-1}(z))\right)\\%
             &=\phi (s)\bullet_2 (\phi (\phi^{-1}(z))%
             =\phi (s)\bullet_2 z.
        \end{align*}
       
        \noindent{}Applying this to $\lambda_{\mathfrak{K}_1}(k)$:
        \begin{equation}
\tag{B}
\begin{split}
((\Gamma \circ \Psi)(\phi) \circ \lambda_{\mathfrak{K}_1})(k) &= \Psi(\phi) \circ \nu_{\mathfrak{K}_1}\left(S_k\right) \circ \Psi(\phi)^{-1}\\
&= \{\phi(s)\bullet_2 (-)\mid s\in S_k\}.
\end{split}
\end{equation}
    \end{enumerate}

    Since Expression (A) is equal to Expression (B), the diagram commutes.

\end{enumerate}

Since all components $\lambda_{\mathfrak{K}}$ are isomorphisms and the naturality condition is satisfied, the natural transformation $\lambda : 1_{\mathscr{T}\mathbb{ODA}} \Rightarrow \Gamma \circ \Psi$ is a natural isomorphism.
\end{proof}

The following corollary establishes the mutual isomorphism among the three forms of $\mathscr{T}$-based orthomodular dynamic algebras, highlighting their structural equivalence.

\medskip
\begin{corollary}\label{isomut}
  Let $\mathfrak{K} = (K, {\bigsqcup}, {\odot}, {}{^*}, {\sim}, e)$   be a $\mathscr{T}$-based orthomodular dynamic algebra. Then $\mathfrak{K}$, $\big(\mathscr{P}(\mathscr T(\mathbf{Lin}(\widetilde{\mathfrak{K}}))),%
    \bigcup, \odot, {}^*, {\sim}, \{\mathrm{id}_M\}\big)$ and $\bigl(\mathscr{P}(\mathscr{T}(K)), \bigcup, \sqdot, {}^{\star}, \sqtilde, \{e\})$ are mutually isomorphic $\mathscr{T}$-based orthomodular dynamic algebras.
\end{corollary}
\medskip
The preceding analysis yields the following conclusions.
\medskip

\begin{theorem}The quadruple $\left(\Gamma, \Psi, \mu, \lambda\right)$ establishes a categorical equivalence between $\mathbb{COL}$ and $\mathscr{T}\mathbb{ODA}$.\end{theorem}

\begin{proof}To demonstrate the categorical equivalence between $\mathbb{COL}$ and $\mathscr{T}\mathbb{ODA}$, it suffices to show the existence of two functors, $\Gamma: \mathbb{COL} \to \mathscr{T}\mathbb{ODA}$ and $\Psi: \mathscr{T}\mathbb{ODA} \to \mathbb{COL}$, together with two natural isomorphisms, $\mu: 1_{\mathbb{COL}} \stackrel{\cong}{\Rightarrow} \Psi \circ \Gamma$ and $\lambda: 1_{\mathscr{T}\mathbb{ODA}} \stackrel{\cong}{\Rightarrow} \Gamma \circ \Psi$.

The functoriality of $\Gamma: \mathbb{COL} \to \mathscr{T}\mathbb{ODA}$ is established by Theorem \ref{Func COL to TODA}. Similarly, $\Psi: \mathscr{T}\mathbb{ODA} \to \mathbb{COL}$ is a functor, as proved in Theorem \ref{thm:G-is-functor}. Furthermore, Theorem \ref{thm:tau-natural} demonstrates that $\mu$ is a natural isomorphism from the identity functor $1_{\mathbb{COL}}$ to the composite functor $\Psi \circ \Gamma$. Dually, Theorem \ref{NI Lambda} confirms that $\lambda$ is a natural isomorphism from $1_{\mathscr{T}\mathbb{ODA}}$ to $\Gamma \circ \Psi$. Since these functors and the required natural isomorphisms exist, the quadruple $\left(\Gamma, \Psi, \mu, \lambda\right)$ fulfills the defining conditions of a categorical equivalence.\end{proof}

\begin{example}\rm 
Since the concrete functor $\mathscr{T}$ from Subsection 
\ref{sec:concrete-T} satisfies conditions ({\bf T1})-({\bf T4}) from Section \ref{sectionTODA}, we obtain a categorical equivalence between the orthomodular semi-Foulis dynamic algebras (which form a subcategory of the orthomodular dynamic algebras introduced by \cite{KRSZ}) and complete orthomodular lattices. Moreover, their basic example from \cite[Section 3]{KRSZ} is, up to isomorphism, exactly our $\mathscr{T}$-based orthomodular dynamic algebra $\mathscr{P}(\mathscr{T}(\mathbf{Lin}(\mathcal{M})))$ from Proposition \ref{prop:T-axioms-hold}.
\end{example}

\medskip 

It is worth mentioning that every orthomodular dynamic algebra from \cite{KRSZ} can be turned into a $\mathscr{T}$-based orthomodular dynamic algebra by defining the involution on the generators as the identity and extending it freely using multiplication and join.

\section*{Conclusion}

This paper establishes a categorical equivalence between the category $\mathbb{COL}$ of complete orthomodular lattices (COL) and the category $\mathscr{T}\mathbb{ODA}$ of $\mathscr{T}$-based orthomodular dynamic algebras. By constructing explicit functors and natural isomorphisms, we provide a clear and rigorous bridge between static quantum logic structures captured by orthomodular lattices and the dynamic quantum actions formalized within these specialized quantale-based algebras.

The equivalence elaborated here generalizes prior results by internalizing the dynamics into the structure of involutive generalized dynamic algebras, with $\mathscr{T}$-based orthomodular dynamic algebras serving as natural counterparts to orthomodular lattices. This correspondence respects orthomodularity, orthocomplementation, and Sasaki projections, preserving the rich algebraic and categorical structures underlying quantum logic.

Beyond theoretical interest, this equivalence offers a robust conceptual framework linking lattice-theoretic and quantale-theoretic approaches to quantum theory. It facilitates transferring results and intuitions between the static and dynamic viewpoints and lays the groundwork for further exploration of quantum-logical phenomena through categorical and algebraic methods.

Future research may extend this framework to specialized lattice classes such as Hilbert lattices, explore computational implementations, and investigate applications to new quantum systems and logic frameworks.

\backmatter

\bmhead{Funding}
This research was funded in part by the Austrian Science Fund (FWF) grant 10.55776/PIN5424624 and the Czech Science Foundation
(GACR) grant 25-20013L.

\end{document}